\documentclass{article}
\usepackage[english]{babel}
\usepackage[utf8]{inputenc}

\usepackage{bm}
\usepackage{amsmath}
\usepackage{graphicx}
\usepackage{color}
\usepackage{subcaption}
\usepackage{amsfonts}
\usepackage{float}
\usepackage{placeins}
\usepackage{algorithm}
\usepackage{algpseudocode}
\usepackage{amsmath}
\usepackage{hyperref}
\usepackage{lineno}
\RequirePackage{booktabs}
\usepackage{verbatim}
\usepackage{multirow} 
\usepackage{dsfont}

\usepackage{amstext}
\usepackage{amsthm}
\usepackage{amssymb}

\DeclareMathOperator*{\argmin}{arg\,min}

\graphicspath{ {./figs/} }

\makeatletter
\numberwithin{equation}{section}
\numberwithin{figure}{section}
\theoremstyle{plain}
\newtheorem{thm}{\protect\theoremname}

\makeatother

\providecommand{\lemmaname}{Lemma}
\providecommand{\theoremname}{Theorem}

\usepackage{authblk}

\makeatother

\title{Multicontinuum splitting schemes for multiscale wave problems}
\date{}

\author[1,2]{Mohsen Alshahrani}
\author[1]{Buzheng Shan}

\affil[1]{Department of Mathematics, Texas A\&M University, College Station, TX, USA.}
\affil[2]{Department of Mathematics, King Fahd University of Petroleum and Minerals, Dhahran,
31261, Saudi Arabia}

\begin{document}

\maketitle

\begin{abstract}
In this work, we propose multicontinuum splitting schemes for the wave equation with a high-contrast coefficient, extending our previous research on multiscale flow problems. 
The proposed approach consists of two main parts: decomposing the solution space into distinct components, and designing tailored time discretization schemes to enhance computational efficiency.
To achieve the decomposition, we employ a multicontinuum homogenization method to introduce physically meaningful macroscopic variables and to separate fast and slow dynamics, effectively isolating contrast effects in high-contrast cases. This decomposition enables the design of schemes where the fast-dynamics (contrast-dependent) component is treated implicitly, while the slow-dynamics (contrast-independent) component is handled explicitly. The idea of discrete energy conservation is applied to derive the stability conditions, which are contrast-independent with appropriately chosen continua. 
We further discuss strategies for optimizing the space decomposition. These include a Rayleigh quotient problem involving tensors, and an alternative generalized eigenvalue decomposition to reduce computational effort.
Finally, various numerical examples are presented to validate the accuracy and stability of our proposed method. 
\end{abstract}

\section{Introduction}

The multiscale wave equation appears in various scientific and engineering applications, such as modeling wave propagation in heterogeneous media, composite materials, and subsurface environments \cite{virieux1984sh,virieux1986p,abdulle2017multiscale}. The variation in material properties, often spanning several orders of magnitude, poses substantial challenges for numerical simulations. In particular, resolving the multiple scales and high contrast typically necessities extremely fine spatial discretization, leading to prohibitive computational costs. Additionally, explicit methods, which are widely used for the time discretization of wave equations, are constrained by stability conditions that require small time step sizes in the presence of high contrast.
To overcome these difficulties, we propose multicontinuum splitting schemes for multiscale wave problems. These schemes can effectively reduce computational costs, and establish stability conditions which are independent of the contrast under certain conditions.

Numerous methods have been developed to tackle multiscale problems and balance the need for accuracy and efficiency. 
Some approaches, such as homogenization methods \cite{papanicolau1978asymptotic,wu2002analysis,owhadi2008numerical,hornung2012homogenization}, introduce effective properties to capture fine-scale details of the media and approximate the original microscopic equation by a macroscopic one. 
Numerical multiscale methods, on the other hand, include the Multiscale Finite Element Method (MsFEM) \cite{hou1997multiscale, hou1999convergence, efendiev2009multiscale,jiang2010analysis}, the Localized Orthogonal Decomposition (LOD) \cite{maalqvist2014localization,henning2014localized,abdulle2017localized}, the Generalized Multiscale Finite Element Method (GMsFEM) \cite{efendiev2013generalized,chung2014generalized,gao2015generalized}, the Constraint Energy Minimizing GMsFEM (CEM-GMsFEM) \cite{chung2018constraint,cheung2021explicit,chung2023multiscale}, the Heterogenuous Multiscale Method (HMM) \cite{weinan2007heterogeneous,abdulle2012heterogeneous,abdulle2011finite}, the Variational Multiscale Method (VMS) \cite{hughes1998variational}, the Nonlocal Multicontinua Method (NLMC) \cite{chung2018non,vasilyeva2019nonlocal}, hybrid approaches that integrate machine learning techniques \cite{wang2020deep,efendiev2023hybrid,rudikov2025locally}, etc. Most of these methods construct multiscale spaces spanned by basis functions to approximate the solution efficiently. A key distinction among them lies in how the basis functions, or the local cell problems used to derive them, are formulated. For example, MsFEM constructs local homogeneous equations to capture fine-scale information, GMsFEM designs spectral problems to identify dominant modes, and CEM-GMsFEM incorporates the idea of energy minimization.

Our approach is built on the framework of multicontinuum homogenization \cite{efendiev2023multicontinuum,chung2024multicontinuum,ammosov2025multicontinuum,xie2025multicontinuum} and inspired by the partially explicit time discretization approach \cite{chung2021contrast,chung2022contrast,chung2021contrast2}.
In multicontinuum homogenization, the solution is represented as an expansion involving macroscopic variables and multiscale basis functions. Unlike classical homogenization, which typically uses a single macroscopic variable, multiple macroscopic variables are defined to represent local averages of the solution within each continuum. The multiscale basis functions are constructed by solving carefully formulated local energy minimization problems on a fine grid. Assuming that the macroscopic variables are smooth over coarse regions, one can derive a system of macroscopic equations, which will be solved on the coarse grid to approximate the solution.
The partially explicit time discretization approach is a methodology specialized for time-dependent problems. This approach utilizes methods like CEM-GMsFEM to select the multiscale modes. It treats the dominant modes implicitly to ensure the stability, while handling the remaining modes explicitly to reduce computational cost. With a suitable selection of modes, a contrast-independent stability can be achieved.

In this work, we propose multicontinuum splitting schemes for the wave equation with a high-contrast coefficient.
In our approach, multicontinuum homogenization is applied to homogenize and discretize in space to address the spatial multiscale properties, while the idea of splitting methods is employed to discretize in time and handle the varying scales of dynamics in the system.
With the framework of multicontinuum homogenization, we categorize the macroscopic variables into two groups: one capable of describing the fast dynamics (typically associated with continua in high-value regions), and the other describing the slow dynamics (typically associated with continua in low-value regions). Then by using the multicontinuum expansion, we decompose the solution space into two components corresponding to the two groups. This enables the separation of fast and slow dynamics, which are often present in high-contrast problems, where different parts of the system evolve at significantly different rates. 
Based on this decomposition, we construct partially explicit time discretization schemes. The component associated with fast dynamics is handled in an implicit fashion to ensure the stability, while the component associated with slow dynamics is processed explicitly to improve efficiency. A key benefit of these schemes is that the derived stability conditions depend solely on the explicitly treated component. Furthermore, this component is independent of the contrast of the coefficient, as long as the continua are chosen appropriately. This feature indicates that the decomposition indeed isolates the effects of high contrast in this case, resulting in stability conditions that are independent of the contrast. 
In addition, we discuss an optimized decomposition of the solution space, where the decomposition is not predefined as above but instead arises from the mixture of continua. This provides more flexibility in separating dynamics with different speeds, and can relax the stability conditions on the time step size without increasing the computation cost. To achieve this, we formulate a generalized Rayleigh quotient problem in tensor form. We also note that it can be simplified into a more computationally efficient generalized eigenvalue problem under certain assumptions.

Numerical results are presented to demonstrate the accuracy and stability of the proposed schemes. Various coefficient fields and different numbers of continua are considered in the experiments. The results show that the proposed space decomposition effectively separates the slow modes from the overall dynamics. With this decomposition, the multicontinuum splitting schemes achieve a comparable level of accuracy to the implicit scheme, while requiring less computational effort.

This research is an extension of our previous work on multiscale flow problems \cite{efendiev2024multicontinuum} to wave problems. First, we design new time discretization schemes tailored for multiscale wave equations, which involves three time layers instead of two, as used in parabolic equations.
Second, we introduce the concept of discrete energy for our schemes, and show its conservation property, as energy conservation is crucial in wave problems.
Additionally, a proof of the stability is provided based on this property to support the methodology.

In Section \ref{sec:preliminaries}, we introduce the necessary background and review the key concepts of multicontinuum homogenization. In Section \ref{sec:schemes}, we present the multicontinuum splitting schemes and derive the corresponding stability conditions. Section \ref{sec:construction} discusses the optimized construction of space decomposition. Numerical results are provided in Section \ref{sec:numerical}, followed by conclusions in Section \ref{sec:conclusions}.

\section{Preliminaries}
\label{sec:preliminaries}

In this work, we consider the wave equation with a high-contrast positive coefficient $\kappa$ defined in a bounded domain $\Omega \subset \mathbb{R}^d$:
\begin{equation}
\label{eq:strong_form}
\begin{split}
\partial_{tt} u(x,t) - \nabla \cdot (\kappa(x) \nabla u(x,t)) = f(x,t), \quad x \in \Omega, \quad t\in (0,T],
\end{split}
\end{equation}
subject to the boundary and initial conditions
\begin{equation}
\begin{split}
    &u(x,t) = g(x), \quad x \in \partial \Omega, \quad t\in (0,T], \\
    &u(x,0) = u_0(x), \quad \partial_{t} u(x,0) = v_0(x),\quad x \in \Omega.
\end{split}
\end{equation}
For simplicity, we assume zero Dirichlet boundary condition throughout this work. 
The weak form of the problem is to find $u(\cdot,t) \in V$ for $0 < t \leq T$ such that
\begin{equation}
\label{eq:variational_form}
(\partial_{tt}u,v)+a(u,v)=(f,v), \quad \forall v \in V,
\end{equation}
where $V=H_0^1(\Omega)$, $a(u,v)=\int_{\Omega} \kappa \nabla u \cdot \nabla v$, and $(\cdot,\cdot)$ denotes the $L^2$ inner product.

Next, we briefly review the multicontinuum homogenization method \cite{efendiev2023multicontinuum,chung2024multicontinuum} and derive the space-homogenized equation for the wave equation. The multicontinuum splitting schemes presented in the subsequent section will be designed according to it.

We start by dividing the spatial domain $\Omega$ into regular coarse-grid elements, $K$, which are not fine enough to resolve the heterogeneities. Within each coarse element $K$, we assume there are $N$ continua, each of which can be described by a function $\psi_i$. The homogenized solution $U_i$ is defined as 
\begin{equation}
U_i(x_K^*)=\dfrac{\int_K u \psi_i}{\int_K \psi_i},
\end{equation}
for some $x_K^*\in K$.
In high-contrast cases, each continuum can typically be identified by the coefficient value as a region within the coarse element, and $\psi_i$ is often chosen as the characteristic function of continuum $i$, given by
\begin{equation}
\psi_i = \begin{cases}
  1, & \text{in continuum} ~i,\\
  0, & \text{otherwise}.
\end{cases}
\end{equation} 
Then $U_i$ represents the local averages of the solution $u$ in continuum $i$. In more general scenarios, $\psi_i$ can be derived by solving local spectral problems to identify its corresponding continuum. We note that the continua are assumed to be given in the rest of this work, as their identification is not the main focus of our study. 
We refer to $U_i$'s as macroscopic variables, and impose an assumption that $U_i$'s are smooth across all coarse elements $K$'s. This smoothness allows us to introduce a two-term multicontinuum expansion \cite{efendiev2023multicontinuum, leung2024some} for the solution $u$ in each coarse element $K$
\begin{equation}
\label{eq:mc_expansion}
u \approx \phi_i U_i + \phi_i^m \nabla_m U_i,
\end{equation}
where summation over repeated indices is implied.

To obtain the multiscale basis functions $\phi_i$ and $\phi_i^m$, we formulate two cell problems, which minimize the local energy in an oversampled region under certain constraints. 
The first cell problem imposes constraints to represent the constants in the average behavior of each continuum
\begin{equation}
\label{eq:mc_cell_problem_1}
\begin{split}
&\int_{K^+}\kappa \nabla \phi_{i} \cdot \nabla v -
\sum_{j,p} \dfrac{\mu_{ij}^p}{\int_{K^p}\psi_j^p} \int_{K^p}\psi_j^p  v = 0,\\
&\int_{K^p} \phi_{i} \psi_j^p = \delta_{ij} \int_{K^p} \psi_j^p.
\end{split}
\end{equation}
The second cell problem imposes constraints to represent the linear functions in the average behavior of each continuum
\begin{equation}
\label{eq:mc_cell_problem_2}
\begin{split}
&\int_{K^+}\kappa \nabla \phi^{m}_{i}\cdot \nabla v -  \sum_{j,p} \dfrac{\mu_{ij}^{mp}}{\int_{K^p}\psi_j^p} \int_{K^p}\psi_j^p v =0,\\
&\int_{K^p}  \phi^{m}_{i} \psi_j^p = \delta_{ij} \int_{K^p} (x_m-\tilde{x}_{m})\psi_j^p ,\\
&\int_{K^{p_0}} (x_m-\tilde{x}_{m})\psi_j^{p_0} =0, \\
\end{split}
\end{equation}
where $\tilde{x}_m$ is a constant.
Here, $\mu$'s are Lagrange multipliers, and $K^+$ is an oversampled region consisting of several coarse blocks, denoted by $K^p$, with the center coarse block being $K=K^{p_0}$. The introduction of the oversampled region can mitigate boundary effects \cite{hou1997multiscale,chung2018constraint}.

Then we homogenize the wave equation in space by applying the multicontinuum expansion for the solution $u$ and the test function $v$
\begin{equation}
u \approx \phi_i U_i + \phi_i^m \nabla_m U_i, \quad v \approx \phi_j V_j + \phi_j^n \nabla_n V_j,
\end{equation}
and estimating each term in the weak form.
For example, the local integral of the time term can be approximated as
\begin{equation}
\begin{split}
\int_{K} (\partial_{tt} u) v 
& \approx \int_{K} (\phi_i \partial_{tt} U_i + \phi_i^m \partial_{tt} \nabla_m U_i)(\phi_j V_j + \phi_j^n \nabla_n V_j) \\
& \approx \int_{K} (\phi_i \partial_{tt} U_i)(\phi_j V_j) \\
& \approx \left( \int_{K} \phi_i \phi_j \right) (\partial_{tt}U_i) V_j.
\end{split}
\end{equation}
Here, we use the smoothness of $U_i$ and $V_j$, along with the fact that $\|\phi_i\|=O(1)$ and $\|\phi_i^{m}\|=O(H)$, where $H$ is the size of coarse block $K$. 
Estimating the other terms similarly and introducing the effective properties
\begin{equation}
\begin{split}
& \gamma_{ij} = \dfrac{1}{|K|} \int_{K} \phi_j \phi_i, \quad
\alpha_{ij} = \dfrac{1}{|K|} \int_{K} \kappa \nabla \phi_j \cdot \nabla \phi_i, \quad
\alpha_{ij}^{mn} = \dfrac{1}{|K|} \int_{K} \kappa \nabla \phi_j^m \cdot \nabla \phi_i^n, \\
& \alpha_{ij}^{m*} = \dfrac{1}{|K|} \int_{K} \kappa \nabla \phi_j^m \cdot \nabla \phi_i, \quad
\alpha_{ij}^{*m} = \dfrac{1}{|K|} \int_{K} \kappa \nabla \phi_j \cdot \nabla \phi_i^m, \quad
f_j = \dfrac{1}{|K|} \int_{K} f \phi_j,
\end{split}
\end{equation}
we can write the weak form as 
\begin{equation}
\label{eq:mc_derivation_step}
\begin{split}
& \int_{\Omega} (\partial_{tt} u) v + \int_{\Omega} \kappa \nabla u \cdot \nabla v \\
\approx & \int_{\Omega} \gamma_{ji} (\partial_{tt} U_i) V_j + \int_{\Omega} \alpha_{ji} U_i V_j + \int_{\Omega} \alpha_{ji}^{*n} U_i \nabla_n V_j + \int_{\Omega} \alpha_{ji}^{m*} \nabla_m U_i V_j + \int_{\Omega} \alpha_{ji}^{mn} \nabla_m U_i \nabla_n V_j \\
\approx & \int_{\Omega} \gamma_{ji} (\partial_{tt} U_i) V_j + \int_{\Omega} \alpha_{ji} U_i V_j + \int_{\Omega} \alpha_{ji}^{mn} \nabla_m U_i \nabla_n V_j \\
\end{split}
\end{equation}
and
\begin{equation}
\begin{split}
\int_{\Omega} f v 
\approx \int_{\Omega} f_j V_j.
\end{split}
\end{equation}
Finally, the homogenized wave equations can be obtained as
\begin{equation}
\gamma_{ji} \partial_{tt} U_i + \alpha_{ji} U_i - \nabla_n (\alpha_{ji}^{mn} \nabla_m U_i) = f_j,
\end{equation}
for $j=1,2,\ldots, N$. For numerical computation, we retain the weak form of these macroscopic equations and solve them on the coarse grid for efficiency.

\section{Multicontinuum splitting schemes}
\label{sec:schemes}

In this section, we discuss temporal splitting schemes for the wave equation.
Building on the multicontinuum homogenization framework introduced in Section \ref{sec:preliminaries}, we design partially explicit time discretization schemes, and derive their stability conditions. Throughout this work, 

We begin by summarizing the semi-discretization in space for the wave equation. We select a finite-dimensional space $V_H \subset V$ constructed upon the coarse mesh $\mathcal{T}_H$, and introduce the discrete multicontinuum space
\begin{equation}
\begin{split}
V_{mc,H}=\left\{ \sum_{x_{l} \in I_H} \mathds{1}_{K_{H}(x_l)} \sum_{i=1}^{N}  \left( \phi_{i}^{K_{H}(x_l)} U_{i} + \tilde{\phi}_{i}^{K_{H}(x_l)}\cdot\nabla U_{i} \right): \; U_i \in V_{H} \right\}.
\end{split}
\end{equation}
Here, $I_{H}$ is the set of centers of all coarse block in $\mathcal{T}_H$, $K_{H}(x_l)$ represents the coarse block with center $x_l$, and $\mathds{1}_{K_{H}(x_l)}$ denotes its corresponding characteristic function. The multiscale basis functions $\phi_{i}^{K_{H}(x_l)}$ and $\tilde{\phi}_{i}^{K_{H}(x_l)} = (\phi_{i}^{K_{H}(x_l),m})_m$ are obtained by solving the cell problems (\ref{eq:mc_cell_problem_1}) and (\ref{eq:mc_cell_problem_2}), with $K$ replaced by $K_{H}(x_l)$. The number of continua is denoted by $N$.
Then for the spatial discretization, we aim to find $u(\cdot,t) \in V_{mc,H}$ such that 
\begin{equation}
\label{eq:mc_variational}
(\partial_{tt}u,v)+a(u,v)=(f,v), \quad \forall v \in V_{mc,H},
\end{equation}
with the initial conditions projected into $V_{mc,H}$.

Next, we consider decomposing the multicontinuum space $V_{mc,H}$ into the direct sum of two subspaces. Suppose $I_1$ and $I_2$ are two disjoint subsets of $\{1,2,\ldots,N\}$ satisfying $I_1 \sqcup I_2=\{1,2,\ldots,N\}$, and define
\begin{equation}
\label{eq:first_decomposition}
\begin{split}
V_{mc,j,H} = \left\{ \sum_{x_{l} \in I_H} \mathds{1}_{K_{H}(x_l)} \sum_{i\in I_{j}} \left( \phi_{i}^{K_{H}(x_l)} U_{i} + \tilde{\phi}_{i}^{K_{H}(x_l)}\cdot\nabla U_{i} \right): \; U_i \in V_{H} \right\}.
\end{split}
\end{equation}
It follows that $V_{mc,H}=V_{mc,1,H} \oplus V_{mc,2,H}$. For convenience, we introduce the downscaling operator $T_{j}: \; \prod_{i\in I_j} V_{H} \rightarrow V_{mc,j,H}$ and its constant and linear components $T_{j,0},T_{j,1}$ by
\begin{equation}
\begin{split}
    T_{j,0}((U_i)_{i\in I_{j}}) & = \sum_{x_{l} \in I_H} \mathds{1}_{K_{H}(x_l)} \sum_{i\in I_{j}} \phi_{i}^{K_{H}(x_l)} U_{i}, \\
    T_{j,1}((U_i)_{i\in I_{j}}) & = \sum_{x_{l} \in I_H} \mathds{1}_{K_{H}(x_l)} \sum_{i\in I_{j}} \tilde{\phi}_{i}^{K_{H}(x_l)}\cdot\nabla U_{i},
\end{split}
\end{equation}
and $T_{j}=T_{j,0}+T_{j,1}$, where $j=1,2$. We also define the bilinear forms $m_{ij}(\cdot,\cdot)$, $a_{ij}(\cdot,\cdot)$, $b_{ij}(\cdot,\cdot)$ and $c_{ij}(\cdot,\cdot)$ by
\begin{equation}
\begin{aligned}
m_{ij}(U,W) &= (T_{j}U,T_{i}W),  & a_{ij}(U,W)=a(T_{j,1}U,T_{i,1}W), \\
b_{ij}(U,W) &= a(T_{j,1}U,T_{i,0}W), & c_{ij}(U,W)=a(T_{j,0}U,T_{i,0}W),
\end{aligned}
\end{equation}
for any $U \in \prod_{k\in I_j} V_{H}$ and $W \in \prod_{k\in I_i} V_{H}$.

Note that the semi-discretization (\ref{eq:mc_variational}) is equivalent to considering $u=u_{1}+u_{2}$, with $u_{1} \in V_{mc,1,H}$ and $u_{2} \in V_{mc,2,H}$ such that
\begin{equation}
\label{eq:mc_variational_split_1}
\begin{split}
(\partial_{tt}u_{1},v_{1})+(\partial_{tt}u_{2},v_{1})+a(u_{1},v_{1})+a(u_{2},v_{1}) & =(f,v_{1}), \quad \forall v_1\in V_{mc,1,H},\\
(\partial_{tt}u_{1},v_{2})+(\partial_{tt}u_{2},v_{2})+a(u_{1},v_{2})+a(u_{2},v_{2}) & =(f,v_{2}), \quad \forall v_2\in V_{mc,2,H}.
\end{split}
\end{equation}
It can be further rewritten as seeking $U_1 \in V_{1,H}:= \prod_{k\in I_1} V_{H}$ and $U_2 \in V_{2,H}:= \prod_{k\in I_2} V_{H}$ satisfying
\begin{equation}
\label{eq:eq:mc_variational_split_2}
\begin{split}
\sum_{j=1}^{2}m_{ij}(\partial_{tt}U_{j},W_{i})+\sum_{j=1}^{2}a_{ij}(U_{j},W_{i})+\sum_{j=1}^{2}c_{ij}(U_{j},W_{i}) & =(f_{i},W_{i}),\quad\forall W_{i} \in V_{i,H},
\end{split}
\end{equation}
for $i=1,2$. Here, we write $(f,T_{i}W_{i})=(f_{i},W_{i})$, and use $b_{ij}(U_{j},W_{i})+b_{ji}(W_{i},U_{j})=0$, which can be verified through integration by parts \cite{efendiev2023multicontinuum}.

When discretizing in time, our objective is to design a numerical scheme whose stability condition is independent of one of the two subspaces $V_{mc,1,H}$ and $V_{mc,2,H}$. Consequently, the stability of the scheme will depend solely on the slow dynamics, if all the fast dynamics has been collected in one subspace.
Now we assume that $V_{mc,1,H}$ includes all the fast modes. To achieve our goal, we treat $U_1$ and $U_2$ in an implicit and explicit fashion, respectively, in (\ref{eq:eq:mc_variational_split_2}), and formulate the following three-layer partially explicit time discretization schemes:

\noindent\textbf{Discretization scheme 1:} Find $U_{1}^{n+1}\in V_{1,H}$ and $U_{2}^{n+1}\in V_{2,H}$ such that
\begin{equation}
\label{eq:Scheme_1}
\begin{split}
&\dfrac{1}{\tau^2} m_{11}(U_{1}^{n+1}-2U_{1}^{n}+U_{1}^{n-1},W_{1}) + \dfrac{1}{\tau^2} m_{12}(U_{2}^{n+1}-2U_{2}^{n}+U_{2}^{n-1},W_{1}) \\ 
+&\dfrac{1}{2}a_{11}(U_{1}^{n+1}+U_{1}^{n-1},W_{1})+a_{12}(U_{2}^{n},W_{1})+\dfrac{1}{2}c_{11}(U_{1}^{n+1}+U_{1}^{n-1},W_{1})+\dfrac{1}{2}c_{12}(U_{2}^{n+1}+U_{2}^{n-1},W_{1}) = (f_{1},W_{1}),\quad \forall W_{1}\in V_{1,H},\\
&\dfrac{1}{\tau^2} m_{22}(U_{2}^{n+1}-2U_{2}^{n}+U_{2}^{n-1},W_{2}) + \dfrac{1}{\tau^2} m_{21}(U_{1}^{n+1}-2U_{1}^{n}+U_{1}^{n-1},W_{2}) \\
+&a_{21}(U_{1}^{n},W_{2})+a_{22}(U_{2}^{n},W_{2})+\dfrac{1}{2}c_{21}(U_{1}^{n+1}+U_{1}^{n-1},W_{2})+\dfrac{1}{2}c_{22}(U_{2}^{n+1}+U_{2}^{n-1},W_{2}) = (f_{2},W_{2}),\quad \forall W_{2}\in V_{2,H}.
\end{split}
\end{equation}

\noindent\textbf{Discretization scheme 2:} Find $U_{1}^{n+1}\in V_{1,H}$ and $U_{2}^{n+1}\in V_{2,H}$ such that
\begin{equation}
\label{eq:Scheme_2}
\begin{split}
&\dfrac{1}{\tau^2} m_{11}(U_{1}^{n+1}-2U_{1}^{n}+U_{1}^{n-1},W_{1}) + \dfrac{1}{\tau^2} m_{12}(U_{2}^{n+1}-2U_{2}^{n}+U_{2}^{n-1},W_{1}) \\ 
+&\dfrac{1}{2}a_{11}(U_{1}^{n+1}+U_{1}^{n-1},W_{1})+a_{12}(U_{2}^{n},W_{1})+\dfrac{1}{2}c_{11}(U_{1}^{n+1}+U_{1}^{n-1},W_{1})+c_{12}(U_{2}^{n},W_{1}) = (f_{1},W_{1}),\quad \forall W_{1}\in V_{1,H},\\
&\dfrac{1}{\tau^2} m_{22}(U_{2}^{n+1}-2U_{2}^{n}+U_{2}^{n-1},W_{2}) + \dfrac{1}{\tau^2} m_{21}(U_{1}^{n+1}-2U_{1}^{n}+U_{1}^{n-1},W_{2}) \\
+&a_{21}(U_{1}^{n},W_{2})+a_{22}(U_{2}^{n},W_{2})+c_{21}(U_{1}^{n},W_{2})+c_{22}(U_{2}^{n},W_{2}) = (f_{2},W_{2}),\quad \forall W_{2}\in V_{2,H}.
\end{split}
\end{equation}
These schemes can be rewritten in matrix form, following the steps outlined in Section \ref{sec:preliminaries}, which makes them straightforward for numerical implementation.

The difference between the two schemes lies in how the $c_{ij}$ terms are handled. Scheme 1 processes all $c_{ij}$ terms implicitly, making it more stable than scheme 2, particularly when the reaction term in the partial differential equation, which corresponds to $c_{ij}$ terms, is large.
Also, it can be noted that the $m_{12}$ and $m_{21}$ terms vanish when $V_{mc,1,H}$ and $V_{mc,2,H}$ are $L_2$-orthogonal. In this case, the second equation in discretization scheme 2 becomes independent of $U_{1}^{n+1}$, and therefore scheme 2 is totally decoupled. Although scheme 1 remains coupled through the mass matrix, this coupling can be eliminated employing techniques such as mass lumping.

We now examine the stability of the above two schemes. Since energy conservation is critical for the wave equation, we start by showing the discrete energy conservation. We assume $f_1=f_2=0$ to facilitate the discussion.
To denote the discrete energy, we introduce the following norms 
\begin{equation}
\begin{split}
& \|(W_{1},W_{2})\|_{m}^{2}:=\sum_{i,j=1}^{2} m_{ij} (W_{j},W_{i})=(\sum_{i=1}^{2} T_{i}W_{i}, \sum_{i=1}^{2} T_{i}W_{i}),\\
& \|(W_{1},W_{2})\|_{a}^{2}:=\sum_{i,j=1}^{2} a_{ij} (W_{j},W_{i})=a(\sum_{i=1}^{2} T_{i,1}W_{i}, \sum_{i=1}^{2} T_{i,1}W_{i}),\\
& \|(W_{1},W_{2})\|_{c}^{2}:=\sum_{i,j=1}^{2} c_{ij} (W_{j},W_{i})=a(\sum_{i=1}^{2} T_{i,0}W_{i}, \sum_{i=1}^{2} T_{i,0}W_{i}),
\end{split}
\end{equation}
for any $W_{1} \in V_{1,H}$ and $W_{2} \in V_{2,H}$.

\begin{thm}
\label{thm:scheme_1_conservation}
For the discretization scheme 1 in (\ref{eq:Scheme_1}), we have the discrete energy conservation 
\begin{equation}
E^{n+\frac{1}{2}} = E^{n-\frac{1}{2}},
\end{equation}
where
\begin{equation}
\begin{split}
E^{n+\frac{1}{2}} 
% & = 
% \dfrac{2}{\tau^2} \|(U_{1}^{n+1}-U_{1}^{n},U_{2}^{n+1}-U_{2}^{n})\|_m^2 
% + \|(U_{1}^{n+1},U_{2}^{n+1})\|_c^2 + \|(U_{1}^{n},U_{2}^{n})\|_c^2
% + \sum_{i=1,2} \left( \|U_i^{n+1}\|_{a_{ii}}^2 + \|U_i^{n}\|_{a_{ii}}^2 \right) + \\
% & 2 a_{12}(U_{2}^{n+1},U_{1}^{n}) + 2 a_{21}(U_{1}^{n+1},U_{2}^{n}) - \|U_{2}^{n+1}-U_{2}^{n}\|_{a_{22}}^2 \\
= &
\dfrac{2}{\tau^2} \|(U_{1}^{n+1}-U_{1}^{n},U_{2}^{n+1}-U_{2}^{n})\|_m^2 
+ \|(U_{1}^{n+1},U_{2}^{n+1})\|_c^2 + \|(U_{1}^{n},U_{2}^{n})\|_c^2 \\
& + \|(U_{1}^{n+1},U_{2}^{n})\|_a^2 + \|(U_{1}^{n},U_{2}^{n+1})\|_a^2
- \|U_{2}^{n+1}-U_{2}^{n}\|_{a_{22}}^2.
\end{split}
\end{equation}
\end{thm}

\begin{proof}

We choose $W_1=U_{1}^{n+1}-U_{1}^{n-1}$ and $W_2=U_{2}^{n+1}-U_{2}^{n-1}$ in discretization scheme 1 in (\ref{eq:Scheme_1}), and we have
\begin{equation}
\label{eq:Thm_1_main_eqs}
\begin{split}
&\dfrac{1}{\tau^2} m_{11}(U_{1}^{n+1}-2U_{1}^{n}+U_{1}^{n-1},U_{1}^{n+1}-U_{1}^{n-1}) + \dfrac{1}{\tau^2} m_{12}(U_{2}^{n+1}-2U_{2}^{n}+U_{2}^{n-1},U_{1}^{n+1}-U_{1}^{n-1}) \\ 
+&\dfrac{1}{2}a_{11}(U_{1}^{n+1}+U_{1}^{n-1},U_{1}^{n+1}-U_{1}^{n-1})+a_{12}(U_{2}^{n},U_{1}^{n+1}-U_{1}^{n-1}) \\
+&\dfrac{1}{2}c_{11}(U_{1}^{n+1}+U_{1}^{n-1},U_{1}^{n+1}-U_{1}^{n-1})+\dfrac{1}{2}c_{12}(U_{2}^{n+1}+U_{2}^{n-1},U_{1}^{n+1}-U_{1}^{n-1}) = 0,\\
&\dfrac{1}{\tau^2} m_{22}(U_{2}^{n+1}-2U_{2}^{n}+U_{2}^{n-1},U_{2}^{n+1}-U_{2}^{n-1}) + \dfrac{1}{\tau^2} m_{21}(U_{1}^{n+1}-2U_{1}^{n}+U_{1}^{n-1},U_{2}^{n+1}-U_{2}^{n-1}) \\
+&a_{21}(U_{1}^{n},U_{2}^{n+1}-U_{2}^{n-1})+a_{22}(U_{2}^{n},U_{2}^{n+1}-U_{2}^{n-1})\\
+&\dfrac{1}{2}c_{21}(U_{1}^{n+1}+U_{1}^{n-1},U_{2}^{n+1}-U_{2}^{n-1})+\dfrac{1}{2}c_{22}(U_{2}^{n+1}+U_{2}^{n-1},U_{2}^{n+1}-U_{2}^{n-1}) = 0.
\end{split}
\end{equation}

First, we rewrite the terms involving the bilinear forms $m_{ij}$ as
\begin{equation}
\begin{split}
& m_{11}(U_{1}^{n+1}-2U_{1}^{n}+U_{1}^{n-1},U_{1}^{n+1}-U_{1}^{n-1}) \\
= & m_{11}((U_{1}^{n+1}-U_{1}^{n})-(U_{1}^{n}-U_{1}^{n-1}),(U_{1}^{n+1}-U_{1}^{n})+(U_{1}^{n}-U_{1}^{n-1})) \\
= & m_{11}(U_{1}^{n+1}-U_{1}^{n},U_{1}^{n+1}-U_{1}^{n}) - m_{11}(U_{1}^{n}-U_{1}^{n-1},U_{1}^{n}-U_{1}^{n-1}),
\end{split}
\end{equation}
\begin{equation}
\begin{split}
& m_{22}(U_{2}^{n+1}-2U_{2}^{n}+U_{2}^{n-1},U_{2}^{n+1}-U_{2}^{n-1}) \\
= & m_{22}(U_{2}^{n+1}-U_{2}^{n},U_{2}^{n+1}-U_{2}^{n}) - m_{22}(U_{2}^{n}-U_{2}^{n-1},U_{2}^{n}-U_{2}^{n-1}),
\end{split}
\end{equation}
and
\begin{equation}
\begin{split}
& m_{12}(U_{2}^{n+1}-2U_{2}^{n}+U_{2}^{n-1},U_{1}^{n+1}-U_{1}^{n-1}) + m_{21}(U_{1}^{n+1}-2U_{1}^{n}+U_{1}^{n-1},U_{2}^{n+1}-U_{2}^{n-1}) \\
= & m_{12}(U_{2}^{n+1}-U_{2}^{n},U_{1}^{n+1}-U_{1}^{n}) - m_{12}(U_{2}^{n}-U_{2}^{n-1},U_{1}^{n}-U_{1}^{n-1}) + \\
& m_{21}(U_{1}^{n+1}-U_{1}^{n},U_{2}^{n+1}-U_{2}^{n}) - m_{21}(U_{1}^{n}-U_{1}^{n-1},U_{2}^{n}-U_{2}^{n-1}).
\end{split}
\end{equation}
Adding them up gives us
\begin{equation}
\begin{split}
& m_{11}(U_{1}^{n+1}-2U_{1}^{n}+U_{1}^{n-1},U_{1}^{n+1}-U_{1}^{n-1}) + m_{22}(U_{2}^{n+1}-2U_{2}^{n}+U_{2}^{n-1},U_{2}^{n+1}-U_{2}^{n-1}) + \\
& m_{12}(U_{2}^{n+1}-2U_{2}^{n}+U_{2}^{n-1},U_{1}^{n+1}-U_{1}^{n-1}) + m_{21}(U_{1}^{n+1}-2U_{1}^{n}+U_{1}^{n-1},U_{2}^{n+1}-U_{2}^{n-1}) \\
= & \sum_{i,j=1}^{2} m_{ij}(U_{j}^{n+1}-U_{j}^{n},U_{i}^{n+1}-U_{i}^{n}) - \sum_{i,j=1}^{2} m_{ij}(U_{j}^{n}-U_{j}^{n-1},U_{i}^{n}-U_{i}^{n-1}) \\
= & \|(U_{1}^{n+1}-U_{1}^{n},U_{2}^{n+1}-U_{2}^{n})\|_{m}^2 - \|(U_{1}^{n}-U_{1}^{n-1},U_{2}^{n}-U_{2}^{n-1})\|_{m}^2.
\end{split}
\end{equation}

Next, the summation of $c_{ij}$ terms can be rewritten as
\begin{equation}
\begin{split}
& c_{11}(U_{1}^{n+1}+U_{1}^{n-1},U_{1}^{n+1}-U_{1}^{n-1}) + c_{22}(U_{2}^{n+1}+U_{2}^{n-1},U_{2}^{n+1}-U_{2}^{n-1}) + \\
& c_{12}(U_{2}^{n+1}+U_{2}^{n-1},U_{1}^{n+1}-U_{1}^{n-1}) + c_{21}(U_{1}^{n+1}+U_{1}^{n-1},U_{2}^{n+1}-U_{2}^{n-1}) \\
= & c_{11}(U_{1}^{n+1},U_{1}^{n+1}) - c_{11}(U_{1}^{n-1},U_{1}^{n-1}) + c_{22}(U_{2}^{n+1},U_{2}^{n+1}) - c_{22}(U_{2}^{n-1},U_{2}^{n-1}) + \\
& c_{12}(U_{2}^{n+1},U_{1}^{n+1}) - c_{12}(U_{2}^{n-1},U_{1}^{n-1}) + c_{21}(U_{1}^{n+1},U_{2}^{n+1}) - c_{21}(U_{1}^{n-1},U_{2}^{n-1}) \\
= & \|(U_{1}^{n+1},U_{2}^{n+1})\|_{c}^2 - \|(U_{1}^{n-1},U_{2}^{n-1})\|_{c}^2 \\
= & \left(\|(U_{1}^{n+1},U_{2}^{n+1})\|_{c}^2 + \|(U_{1}^{n},U_{2}^{n})\|_{c}^2\right) - \left(\|(U_{1}^{n},U_{2}^{n})\|_{c}^2 + \|(U_{1}^{n-1},U_{2}^{n-1})\|_{c}^2\right).
\end{split}
\end{equation}

Additionally, for the $a_{ij}$ terms, by using the symmetry of bilinear form $a(\cdot,\cdot)$, we obtain
\begin{equation}
\begin{split}
a_{11}(U_{1}^{n+1}+U_{1}^{n-1},U_{1}^{n+1}-U_{1}^{n-1})
= a_{11}(U_{1}^{n+1},U_{1}^{n+1})-a_{11}(U_{1}^{n-1},U_{1}^{n-1}),
\end{split}
\end{equation}
and
\begin{equation}
\begin{split}
& a_{12}(U_{2}^{n},U_{1}^{n+1}-U_{1}^{n-1}) + a_{21}(U_{1}^{n},U_{2}^{n+1}-U_{2}^{n-1}) \\
= & a_{12}(U_{2}^{n},U_{1}^{n+1}) - a_{12}(U_{2}^{n},U_{1}^{n-1}) + a_{21}(U_{1}^{n},U_{2}^{n+1}) - a_{21}(U_{1}^{n},U_{2}^{n-1}) \\
= & a_{12}(U_{2}^{n+1},U_{1}^{n}) + a_{21}(U_{1}^{n+1},U_{2}^{n}) - a_{12}(U_{2}^{n},U_{1}^{n-1}) - a_{21}(U_{1}^{n},U_{2}^{n-1}),
\end{split}
\end{equation}
and
\begin{equation}
\begin{split}
& a_{22}(U_{2}^{n},U_{2}^{n+1}-U_{2}^{n-1}) \\
= & a_{22}(U_{2}^{n+1},U_{2}^{n}) - a_{22}(U_{2}^{n},U_{2}^{n-1}) \\
= & \dfrac{1}{2} \left( a_{22}(U_{2}^{n+1},U_{2}^{n+1}) + a_{22}(U_{2}^{n},U_{2}^{n}) - a_{22}(U_{2}^{n+1}-U_{2}^{n},U_{2}^{n+1}-U_{2}^{n}) \right) - \\
& \dfrac{1}{2} \left( a_{22}(U_{2}^{n},U_{2}^{n}) + a_{22}(U_{2}^{n-1},U_{2}^{n-1}) - a_{22}(U_{2}^{n}-U_{2}^{n-1},U_{2}^{n}-U_{2}^{n-1}) \right).
\end{split}
\end{equation}
Therefore, the summation of the $a_{ij}$ terms in (\ref{eq:Thm_1_main_eqs}) is
\begin{equation}
\begin{split}
& \dfrac{1}{2} a_{11}(U_{1}^{n+1}+U_{1}^{n-1},U_{1}^{n+1}-U_{1}^{n-1}) + a_{12}(U_{2}^{n},U_{1}^{n+1}-U_{1}^{n-1}) + a_{21}(U_{1}^{n},U_{2}^{n+1}-U_{2}^{n-1}) + a_{22}(U_{2}^{n},U_{2}^{n+1}-U_{2}^{n-1}) \\
= & \bigg( \dfrac{1}{2} \sum_{i=1,2} \left( \|U_i^{n+1}\|_{a_{ii}}^2 + \|U_i^{n}\|_{a_{ii}}^2 \right) + a_{12}(U_{2}^{n+1},U_{1}^{n}) + a_{21}(U_{1}^{n+1},U_{2}^{n}) - \dfrac{1}{2} \|U_{2}^{n+1}-U_{2}^{n}\|_{a_{22}}^2 \bigg) - \\
& \bigg( \dfrac{1}{2} \sum_{i=1,2} \left( \|U_i^{n}\|_{a_{ii}}^2 + \|U_i^{n-1}\|_{a_{ii}}^2 \right) + a_{12}(U_{2}^{n},U_{1}^{n-1}) + a_{21}(U_{1}^{n},U_{2}^{n-1}) - \dfrac{1}{2} \|U_{2}^{n}-U_{2}^{n-1}\|_{a_{22}}^2 \bigg) \\
= & \dfrac{1}{2} \left( \|(U_{1}^{n+1},U_{2}^{n})\|_{a}^2 + \|(U_{1}^{n},U_{2}^{n+1})\|_{a}^2  - \|U_{2}^{n+1}-U_{2}^{n}\|_{a_{22}}^2 \right) - 
\\ &\dfrac{1}{2} \left( \|(U_{1}^{n},U_{2}^{n-1})\|_{a}^2 + \|(U_{1}^{n-1},U_{2}^{n})\|_{a}^2  - \|U_{2}^{n}-U_{2}^{n-1}\|_{a_{22}}^2 \right).  \\
\end{split}
\end{equation}

Finally, adding the two equations in (\ref{eq:Thm_1_main_eqs}) and combining all the above results, we arrive at
\begin{equation}
\begin{split}
E^{n+\frac{1}{2}} = E^{n-\frac{1}{2}}.
\end{split}
\end{equation}

\end{proof}

According to the strengthened Cauchy-Schwarz inequality \cite{aldaz2009strengthened}, we define a constant $\gamma \in (0,1)$ by
\begin{equation}
\gamma = \sup_{W_1 \in V_{1,H}, W_2 \in V_{2,H}} \dfrac{(T_1 W_1, T_2 W_2)}{\|T_1 W_1\| \|T_2 W_2\|}.
\end{equation} 
Then for any $W_1 \in V_{1,H}$ and $W_2 \in V_{2,H}$, we have
\begin{equation}
\begin{split}
m_{12}(W_2,W_1) = m_{21}(W_1,W_2) =  (T_1 W_1, T_2 W_2) \leq \gamma \|T_1 W_1\| \|T_2 W_2\| = \gamma \|W_1\|_{m_{11}} \|W_2\|_{m_{22}}.
\end{split}
\end{equation}
Using the conservation law in Theorem \ref{thm:scheme_1_conservation}, we obtain the following stability result.

\begin{thm}
The discretization scheme 1 in (\ref{eq:Scheme_1}) is stable if the stability condition
\begin{equation}
\tau^2 \leq 2(1-\gamma^2) \inf_{W\in V_{2,H}}\cfrac{\|W\|_{m_{22}}^{2}}{\|W\|_{a_{22}}^{2}}
\end{equation}
is satisfied.
\end{thm}

\begin{proof}
We first note that 
\begin{equation}
\begin{split}
& \|(U_{1}^{n+1}-U_{1}^{n},U_{2}^{n+1}-U_{2}^{n})\|_m^2 \\
= & \|U_{1}^{n+1}-U_{1}^{n}\|_{m_{11}}^2 + \|U_{2}^{n+1}-U_{2}^{n}\|_{m_{22}}^2 + 2 m_{12}(U_{1}^{n+1}-U_{1}^{n},U_{2}^{n+1}-U_{2}^{n}) \\
\geq & \|U_{1}^{n+1}-U_{1}^{n}\|_{m_{11}}^2 + \|U_{2}^{n+1}-U_{2}^{n}\|_{m_{22}}^2 - 2\gamma \|U_{1}^{n+1}-U_{1}^{n}\|_{m_{11}} \|U_{2}^{n+1}-U_{2}^{n}\|_{m_{22}} \\
\geq & (1-\gamma^2) \|U_{2}^{n+1}-U_{2}^{n}\|_{m_{22}}^2.
\end{split}
\end{equation}
Then we can estimate the discrete energy $E^{n+\frac{1}{2}}$ as
\begin{equation}
\begin{split}
E^{n+\frac{1}{2}}
= &
\dfrac{2}{\tau^2} \|(U_{1}^{n+1}-U_{1}^{n},U_{2}^{n+1}-U_{2}^{n})\|_m^2 
+ \|(U_{1}^{n+1},U_{2}^{n+1})\|_c^2 + \|(U_{1}^{n},U_{2}^{n})\|_c^2
+ \|(U_{1}^{n+1},U_{2}^{n})\|_a^2 + \|(U_{1}^{n},U_{2}^{n+1})\|_a^2 \\
& - \|U_{2}^{n+1}-U_{2}^{n}\|_{a_{22}}^2 \\
\geq & \dfrac{2(1-\gamma^2)}{\tau^2} \|U_{2}^{n+1}-U_{2}^{n}\|_{m_{22}}^2 + \|(U_{1}^{n+1},U_{2}^{n+1})\|_c^2 + \|(U_{1}^{n},U_{2}^{n})\|_c^2
+ \|(U_{1}^{n+1},U_{2}^{n})\|_a^2 + \|(U_{1}^{n},U_{2}^{n+1})\|_a^2 \\
& - \|U_{2}^{n+1}-U_{2}^{n}\|_{a_{22}}^2 \\
\geq & (U_{1}^{n+1},U_{2}^{n+1})\|_c^2 + \|(U_{1}^{n},U_{2}^{n})\|_c^2
+ \|(U_{1}^{n+1},U_{2}^{n})\|_a^2 + \|(U_{1}^{n},U_{2}^{n+1})\|_a^2,
\end{split}
\end{equation}
where the stability condition is applied in the last inequality.
Therefore, $E^{n+\frac{1}{2}}$ defines a nonnegative quantity which remains the same in time, and the discretization scheme 1 is stable.
\end{proof}

For scheme 2, similar energy conservation can be established, leading to the following stability result.

\begin{thm}
The discretization scheme 2 in (\ref{eq:Scheme_2}) is stable if the stability condition
\begin{equation}
\tau^2 \leq 2(1-\gamma^2) \inf_{W\in V_{2,H}}\cfrac{\|W\|_{m_{22}}^{2}}{\|W\|_{a_{22}}^{2}+\|W\|_{c_{22}}^{2}}
\end{equation}
is satisfied.
\end{thm}

The remaining question is how to determine the partition sets $I_1$ and $I_2$. For cases such as two-value fields, $I_1$ and $I_2$ can respectively be chosen as the index sets for the continua associated with the high- and low-value regions. It turns out that the $m_{22}$-norm, the $a_{22}$-norm and the $c_{22}$-norm are all independent of the contrast, resulting in contrast-independent stability conditions, thanks to the construction of the cell problems \cite{efendiev2024multicontinuum}. This result also extends to more complex setting. 
A general strategy for determining the partition will be introduced in the next section. 

\section{Optimized decomposition of the solution space}
\label{sec:construction}

In this section, we present an optimized decomposition of the multicontinuum space $V_{mc,H}$ to relax the stability conditions. For detailed explanations and proofs, we refer the reader to \cite{efendiev2024multicontinuum}.

To consider a more general decomposition, we replace $\phi_{i}$ and $\tilde{\phi}_{i}$ in the previous construction (\ref{eq:first_decomposition}) for $V_{mc,1,H}$ and $V_{mc,2,H}$ with their linear combinations. The superscript $K_{H}(x_l)$ is omitted for brevity. We redefine 
\begin{equation}
\label{eq:Reconstruction_v}
\begin{split}
V_{mc,1,H} &= \left\{ \sum_{x_{l} \in I_H} \mathds{1}_{K_{H}(x_l)} \sum_{i_{0}+1 \leq i \leq N} \left( \sum_j v_{i,j} \phi_{j} \sum_j \hat{v}_{i,j} U_j + \sum_j v_{i,j} \tilde{\phi}_{j}\cdot \sum_j \hat{v}_{i,j} \nabla U_j \right): \; U_j \in V_{H} \right\}, \\
V_{mc,2,H} &= \left\{ \sum_{x_{l} \in I_H} \mathds{1}_{K_{H}(x_l)} \sum_{1\leq i \leq i_{0}} \left( \sum_j v_{i,j} \phi_{j} \sum_j \hat{v}_{i,j} U_j + \sum_j v_{i,j} \tilde{\phi}_{j}\cdot \sum_j \hat{v}_{i,j} \nabla U_j \right): \;U_j \in V_{H} \right\},
\end{split}
\end{equation}
where $i_0$ is an integer between $1$ and $N$, $v_i=(v_{i,j})_{1\leq j \leq N} \in \mathbb{R}^N$ are linearly independent, and $\\(v_{i,j})_{N \times N} (\hat{v}_{i,j})_{N \times N} = I_{N\times N}$. It follows that $V_{mc,H}=V_{mc,1,H}+V_{mc,2,H}$. Then the multicontinuum splitting schemes can be similarly formulated using the new basis functions $\hat{\phi}_i = \sum_j v_{i,j} \phi_{j}$ and $\hat{\tilde{\phi}}_i = \sum_j v_{i,j} \tilde{\phi}_{j}$, along with the macroscopic variables $\hat{U}_i=\sum_j \hat{v}_{i,j} U_j$.

Now we aim to formulate an optimization problem to determine the optimized integer $i_0$ and linearly independent vectors $\{v_i\}_{i=1,2,\ldots,N}$ such that the stability conditions can be relaxed. For a given coarse mesh $\mathcal{T}_H$, we define $d \times d$ matrices $A^{kl}$ and an $N \times N$ symmetric matrix $M$ by
\begin{equation}
    A^{kl}_{mn} = \dfrac{1}{|K|} \int_K \kappa \nabla \phi_k^m \cdot \nabla \phi_l^n,
\end{equation}
and
\begin{equation}
M_{kl} = \dfrac{1}{|K|} \int_K  \phi_k \phi_l,
\end{equation}
where $k,l \in \{1,2,...,N\}$, and $K\in \mathcal{T}_H$. Then we define a rank-4 tensor $A$ by $A_{kmln}=A_{mn}^{kl}$. Next, we formulate a Rayleigh quotient problem involving tensors. For $1 \leq i \leq N$, we want to find a subspace $S_i \subset \mathbb{R}^N$ and a number $\lambda_i \in \mathbb{R}$ such that
\begin{equation}
\label{eq:Scheme_1_Rayleigh}
\begin{split}
S_i & = \argmin_{\substack{S\subset\mathbb{R}^N \\ \dim(S)=i}} \max_{v\otimes w\in S\otimes\mathbb{R}^d} 
\dfrac{(A:(v\otimes w)):(v\otimes w)}{((M\otimes I):(v\otimes w)):(v\otimes w)}, \\
\lambda_i & = \min_{\substack{S\subset\mathbb{R}^N \\ \dim(S)=i}} \max_{v\otimes w\in S\otimes\mathbb{R}^d}
\dfrac{(A:(v\otimes w)):(v\otimes w)}{((M\otimes I):(v\otimes w)):(v\otimes w)}.
\end{split}
\end{equation}
Here, $\otimes$ and $:$ represent the tensor product and the double dot product, respectively.
We let $i_{0}$ denote the number of small $\lambda_i$'s, and select an orthonormal basis $\{v_i\}_{1 \leq i \leq N}$ of $\mathbb{R}^N$ with respect to $M$, satisfying $v_i^T M v_j = \delta_{ij}$, such that
\begin{equation}
\label{eq:Scheme_1_v_i_construction}
S_{i_{0}} = \operatorname*{span}_{1 \leq i \leq i_{0}} \{v_i\},\quad
\mathbb{R}^N = S_{i_{0}} \oplus \operatorname*{span}_{i_{0}+1 \leq i \leq N} \{v_i\}.
\end{equation}
It can be shown that $\gamma=0$. 
Here we assume that the decompositions for different coarse blocks are similar for simplicity, allowing the formulated quotient problem to be localized within a single coarse block $K$ rather than being global.

We now briefly highlight several advantages of the above construction.
First, explicit forms of the stability conditions can be given under such construction. For the chosen finite element space $V_{2,H}$, let $C_1$ be a constant independent of the mesh size $H$, such that 
\begin{equation}
\sup_{W\in V_{2,H}}\dfrac{\|W\|_{H^{1}}^{2}}{\|W\|_{L^{2}}^{2}} \leq \dfrac{C_1}{H^{2}}.  
\end{equation}
Then, the right-hand side of the stability condition for scheme 1 closely approximates and converges to $\sqrt{2} C_1^{-1/2} \lambda_{i_0}^{-1/2} H$ as $H$ decreases, i.e., $\tau \lesssim \sqrt{2} C_1^{-1/2} \lambda_{i_0}^{-1/2} H$.
Similarly, the stability condition for scheme 2 can be expressed as $\tau \lesssim \sqrt{2} (C_1 \lambda_{i_0} + \operatorname{max\;eig}(H^2 C,M))^{-1/2} H$, where the $N\times N$ symmetric matrix $C$ is defined by $C_{ij} = \dfrac{1}{|K|} \int_K \kappa \nabla \phi_i \cdot \nabla \phi_j$. Additionally, as discussed in Section \ref{sec:schemes}, the schemes can be decoupled since $\gamma$ is close to zero and vanishes with decreasing $H$. 
Furthermore, the solution to the above optimization problem is independent of the choice of basis for $\operatorname*{span}_{i}\{\psi_i\}$, which implies that the new continua obtained through the space decomposition have physical significance.

In the numerical examples, instead of directly solving the above tensor-based min-max problem, we adopt an eigendecomposition approach to reduce computational effort. We define a rank-reduced tensor $\tilde{A}$ by $\tilde{A}_{kl}=\operatorname*{max\;eig}(A^{kl})$, and consider $\tilde{A}$ as a symmetric matrix of size $N \times N$. We then perform the generalized eigenvalue decomposition of $(\tilde{A},M)$:
\begin{equation}
\label{eq:Scheme_1_eig}
    \tilde{A} v_i= \lambda M v_i,
\end{equation}
and select $i_0$ as the number of small eigenvalues and $\{v_i\}_{1 \leq i \leq N}$ as the $M$-orthonormal eigenvectors of $\tilde{A}$, corresponding to the eigenvalues $\lambda_i$ arranged in ascending order. The same stability results hold under certain assumption. The obtained solution, however, is suboptimal.

\section{Numerical examples}
\label{sec:numerical}

In this section, we present some numerical results. 
As stated before, we consider the wave equation
\begin{equation}
    \partial_{tt} u(x,t) - \nabla \cdot (\kappa(x) \nabla u(x,t)) = f(x,t), \quad x \in \Omega, \quad t\in (0,T],
\end{equation}
with zero Dirichlet boundary condition and zero initial conditions, where $\kappa$ is a high-contrast coefficient.
We assume that the computational domain $\Omega$ is partitioned into a uniform grid of square coarse blocks $K$ of the same size $H$. 
To mitigate boundary effects, we define the oversampled domain $K^{+}$ by extending $K$ by $l=\lceil -2\ln(H) \rceil$ layers of coarse blocks \cite{hou1997multiscale,chung2018constraint}. 
The relative $L^2$ errors of the solution in continuum $i$ at time $t$ is defined as
\begin{equation}
\label{eq:def_relative_error}
    e^{(i)}(t)=\sqrt{ \frac{\sum_K|\frac1{|K|}\int_K U_i(x,t) dx -  \frac{1}{\int_{K}\psi_i(x) dx}\int_{K}u(x,t)\psi_i(x) dx|^2}{\sum_K|\frac{1}{\int_{K}\psi_i(x) dx}\int_{K}u(x,t)\psi_i(x) dx|^2} },
\end{equation}
where $U_i$ is the multiscale solution in continuum $i$, and $u$ is the reference solution computed on the fine grid. 

In the following examples, we set $\Omega = (0,1)^2$ and choose the source term as $f(x,t)=1000 e^{-40 \left( (x_1-0.5)^2+(x_2-0.5)^2 \right)} e^{-40t}$ for $x=(x_1,x_2)\in \Omega$. The time step size is set to $\tau = 10^{-3}$, and the final time is $T=0.05$. The fine mesh size $h$ is fixed at $1/400$.

\subsection{Example 1: Layered field with two continua}

As the first example, we consider a layered coefficient field $\kappa$ given by
\begin{equation}
\kappa(x)=
\begin{cases}
1,&\quad x \in \Omega_1, \\
10^{3}, & \quad x \in \Omega_2,
\end{cases}
\end{equation}
where $\Omega_1$ (blue) and $\Omega_2$ (yellow) are the subregions shown in Figure \ref{fig:Example_1_setting}. The auxiliary functions $\psi_1$ and $\psi_2$ are chosen as the characteristic functions in the regions of low and high values respectively; namely, we define $\psi_1(x) = \mathds{1}_{\Omega_1}(x)$ and $\psi_2(x) = \mathds{1}_{\Omega_2}(x)$ in each coarse block.

Now we apply the generalized eigenvalue decomposition in (\ref{eq:Scheme_1_eig}), as discussed in Section \ref{sec:construction}, to decompose the multicontinuum space and reconstruct the subspaces $V_{mc,1,H}$ and $V_{mc,2,H}$. The results of the eigenvalue problems are listed in Table \ref{tab:Example_1_eigen}. It can be observed that the two eigenvalues differ significantly, and the eigenvector corresponding to the smaller eigenvalue has an almost zero second coordinate. This indicates that the first continuum represents slow dynamics and should be treated explicitly to reduce computational cost, while the other one represents fast dynamics and should be treated implicitly to ensure the stability. These observations are consistent with physical intuition.

With the solution space decomposition following from the above results, we compute numerical solutions for the wave equation under different discretization schemes. We present the relative $L^2$ errors for coarse mesh sizes $H=1/10$ and $H=1/20$ as time progresses in Figures \ref{fig:Example_1_error_1} and \ref{fig:Example_1_error_2}. The plots on the left and right respectively describe the errors in the low- and high-value regions. The reference solution is computed on the fine grid using the three-layer implicit discretization. 
In these figures, the red and yellow curves, representing discretization schemes 1 and 2, nearly coincide with the blue curve, which corresponds to the implicit scheme based on multicontinuum homogenization. In contrast, the explicit scheme produces solutions that are unstable and fail to converge for the given time step size; hence, they are not depicted. These results suggest that our proposed schemes, which treat selected degrees of freedom explicitly, achieve comparable accuracy to the implicit scheme, which treats all the degrees of freedom implicitly. In addition, we can also notice that the accuracy improves as the coarse mesh size $H$ becomes finer.

\begin{figure}
    \centering
    \includegraphics[width=7cm]{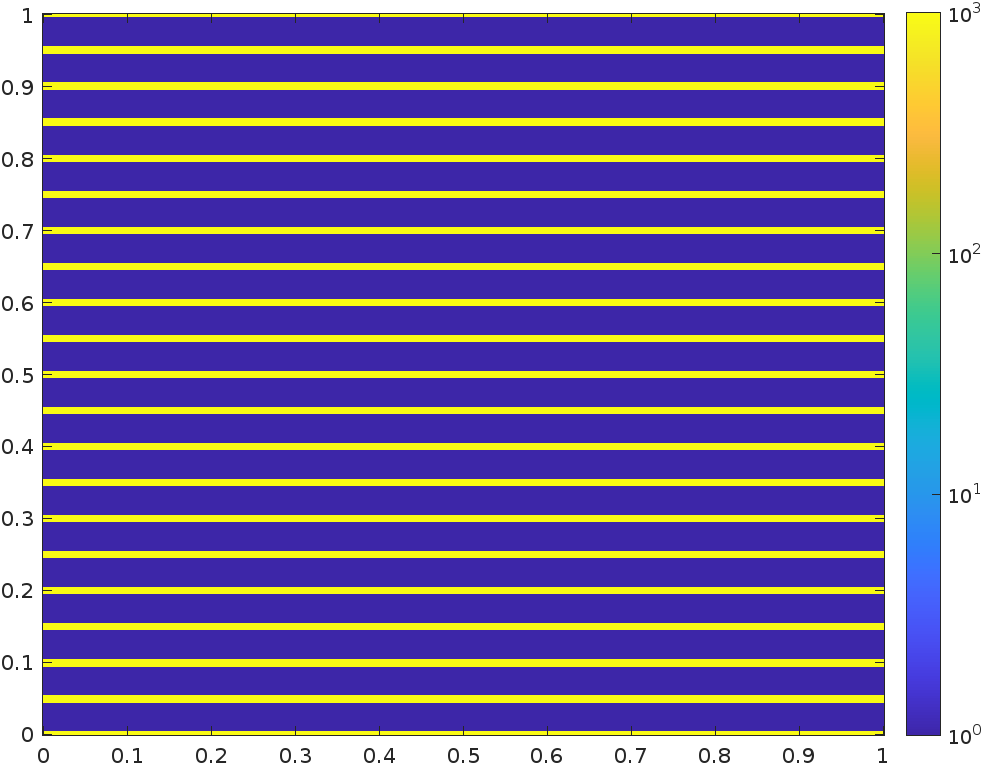}
    \quad
    \includegraphics[width=7cm]{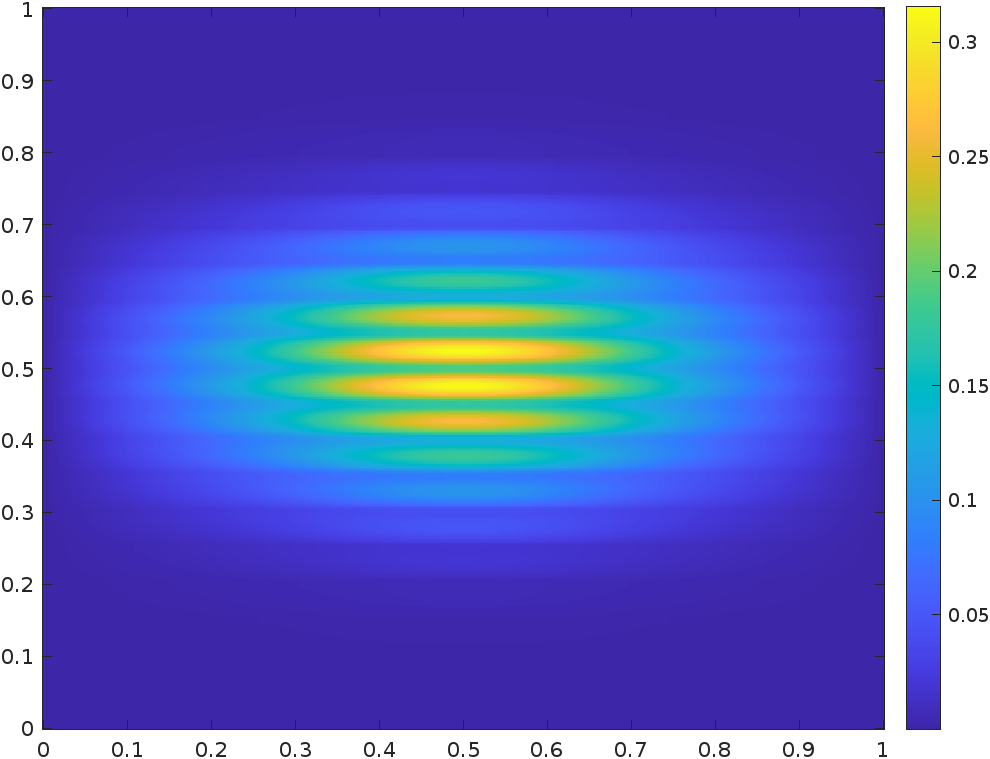}
    \caption{Left: Layered field $\kappa$ in Example 1. Right: Reference solution at the final time $T$ in Example 1.}
    \label{fig:Example_1_setting}
\end{figure}

\begin{table}
    \centering
    \caption{Eigenvalues and corresponding eigenvectors of eigenvalue problems for different coarse mesh sizes $H$ in Example 1.}
    \begin{tabular}{|c|c|c|c|}
         \hline
         $H$ & $l$ & eigenvalues $\lambda_i$ &  eigenvectors $v_i$ \\
         \hline
         \multirow{2}*{1/10} & \multirow{2}*{5} & $3.9487$ & $(1.0195, -0.0032)^T$ \\
         \cline{3-4}
         & & $609.0326$ & $(0.2920, 1.7299)^T$ \\
         \hline
         \multirow{2}*{1/20} & \multirow{2}*{6} & $2.0524$ & $(1.0197, -0.0017)^T$ \\
         \cline{3-4}
         & & $601.8824$ & $(0.2911, 1.7299)^T$ \\
         \hline
    \end{tabular}   
    \label{tab:Example_1_eigen}
\end{table}

\begin{figure}
  \centering
  \includegraphics[width=7cm]{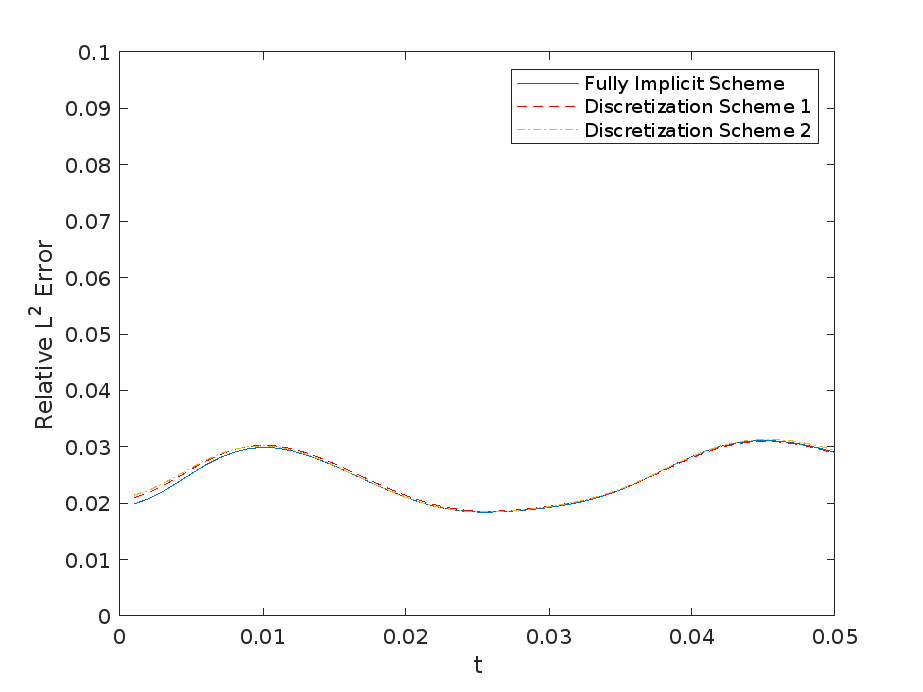}
  \quad
  \includegraphics[width=7cm]{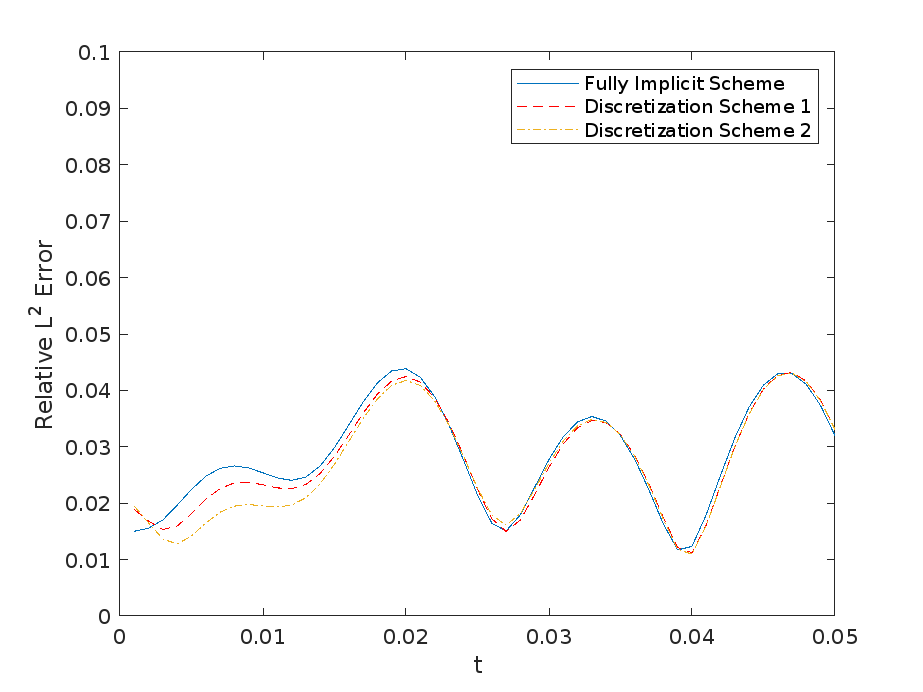}
  \quad
  \caption{Relative $L^2$ error for different schemes when $H=1/10$ and $l=5$ in Example 1. Left: $e^{(1)}(t)$. Right: $e^{(2)}(t)$.}
  \label{fig:Example_1_error_1}
\end{figure}

\begin{figure}
  \centering
  \includegraphics[width=7cm]{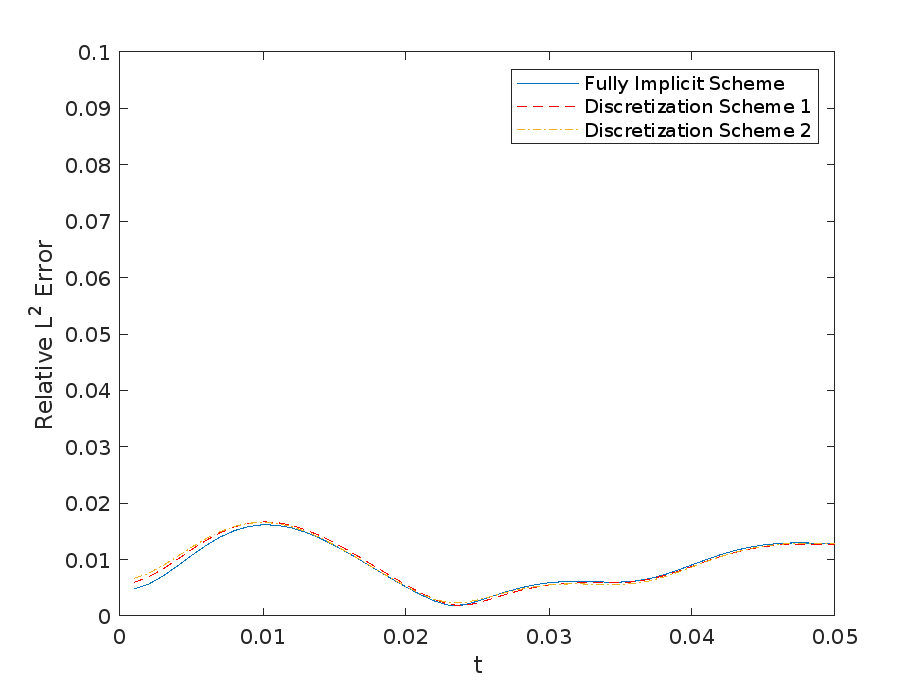}
  \quad
  \includegraphics[width=7cm]{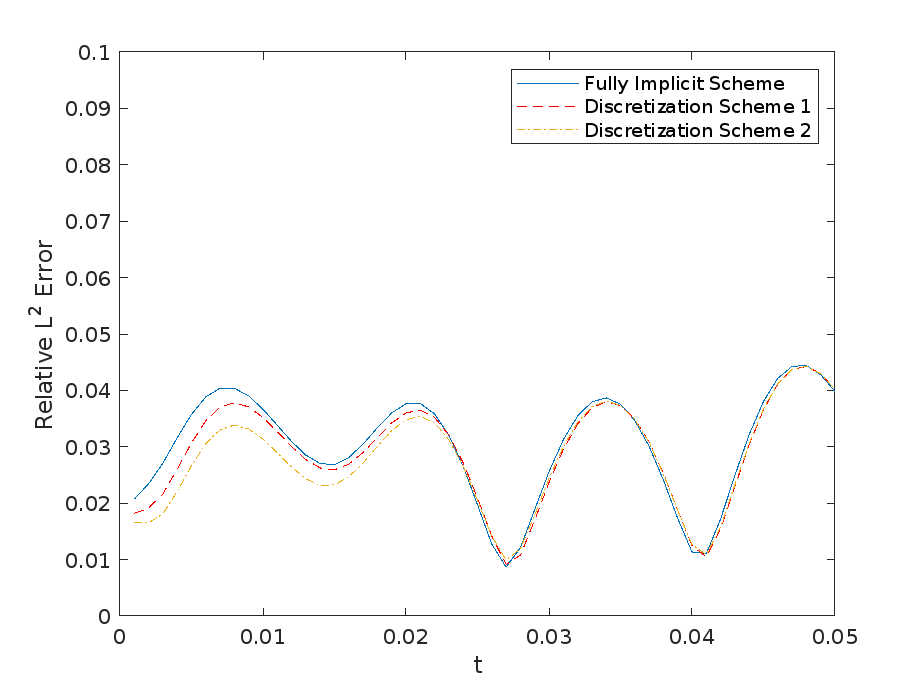}
  \quad
  \caption{Relative $L^2$ error for different schemes when $H=1/20$ and $l=6$ in Example 1. Left: $e^{(1)}(t)$. Right: $e^{(2)}(t)$.}
  \label{fig:Example_1_error_2}
\end{figure}

\subsection{Example 2: Point field with two continua}

In the second example, we take $\kappa$ to be a point field, as illustrated in Figure \ref{fig:Example_2_setting}, with two high-contrast values, $1$ and $10^{3}$. The auxiliary functions $\psi_1$ and $\psi_2$ are again chosen as the characteristic functions in the low- and high-value regions respectively.

We can similarly obtain the results of eigenvalue problems, which are listed in Table \ref{tab:Example_2_eigen}. It can be seen that the eigenvector corresponding to the smaller eigenvalue is still nearly the canonical basis vector $e_1=(1,0)^T$. This indicates that the first continuum is identified as representing slow dynamics and is processed explicitly.

Using the space decomposition described above, we compute the numerical solutions for the wave equation under different discretization schemes. The relative $L^2$ errors for coarse mesh sizes $H=1/10$ and $H=1/20$  are plotted in Figures \ref{fig:Example_2_error_1} and \ref{fig:Example_2_error_2}. The reference solution at the final time $T$ is depicted in Figure \ref{fig:Example_2_setting}. Comparable patterns to those in Example 1 are observed. The solutions obtained using the implicit scheme based on homogenization, as well as discretization schemes 1 and 2, exhibit small and nearly identical errors, all of which closely match the reference solution.

\begin{figure}
    \centering
    \includegraphics[width=7cm]{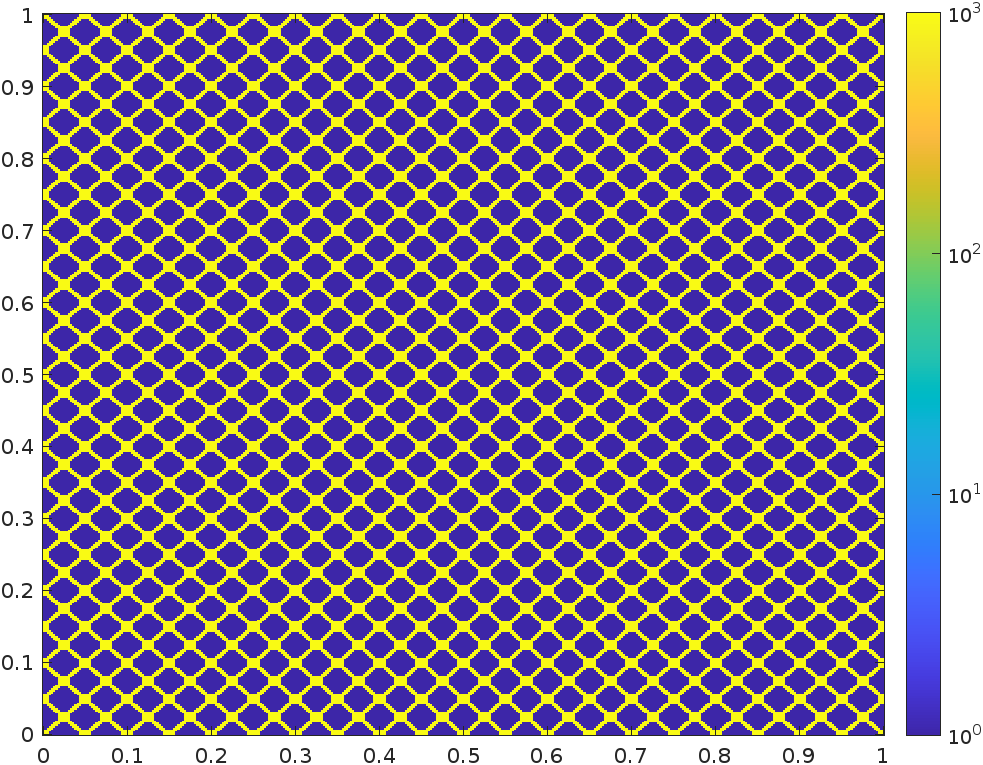}
    \quad
    \includegraphics[width=7cm]{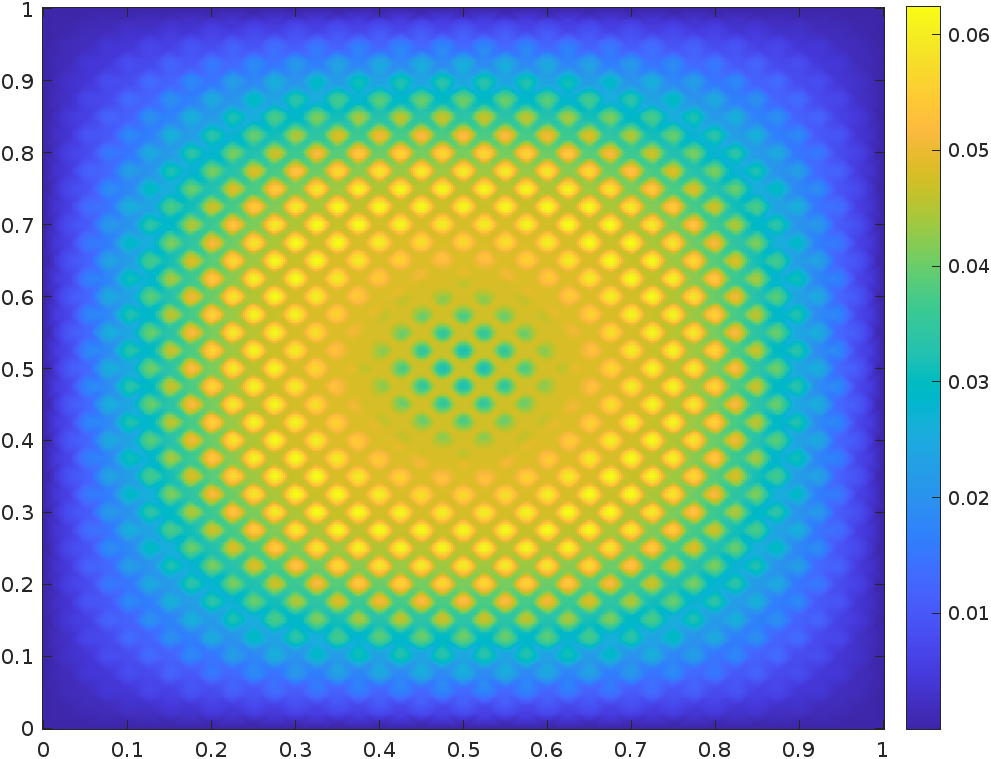}
    \caption{Left: Point field $\kappa$ in Example 2. Right: Reference solution at the final time $T$ in Example 2.}
    \label{fig:Example_2_setting}
\end{figure}

\begin{table}
    \centering
    \caption{Eigenvalues and corresponding eigenvectors of eigenvalue problems for different coarse mesh sizes $H$ in Example 2.}
    \begin{tabular}{|c|c|c|c|}
         \hline
         $H$ & $l$ & eigenvalues $\lambda_i$ &  eigenvectors $v_i$ \\
         \hline
         \multirow{2}*{1/10} & \multirow{2}*{5} & $9.6769$ & $(1.0244, 0.0321)^T$ \\
         \cline{3-4}
         & & $362.4629$ & $(0.3900, 1.3817)^T$ \\
         \hline
         \multirow{2}*{1/20} & \multirow{2}*{6} & $4.9011$ & $(1.0191, 0.0137)^T$ \\
         \cline{3-4}
         & & $356.7905$ & $(0.4036, 1.3820)^T$ \\
         \hline
    \end{tabular}   
    \label{tab:Example_2_eigen}
\end{table}

\begin{figure}
  \centering
  \includegraphics[width=7cm]{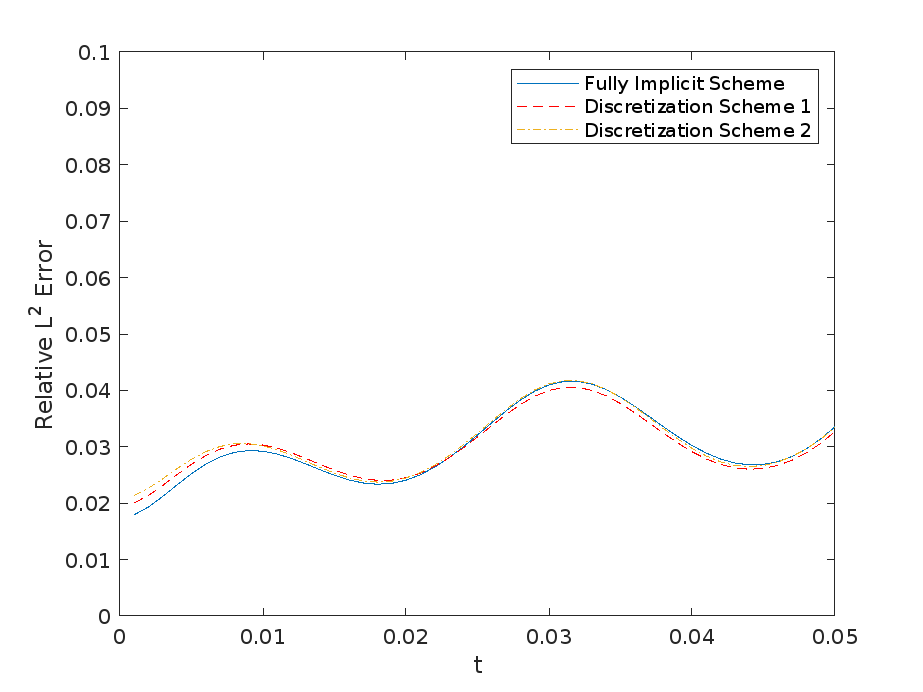}
  \quad
  \includegraphics[width=7cm]{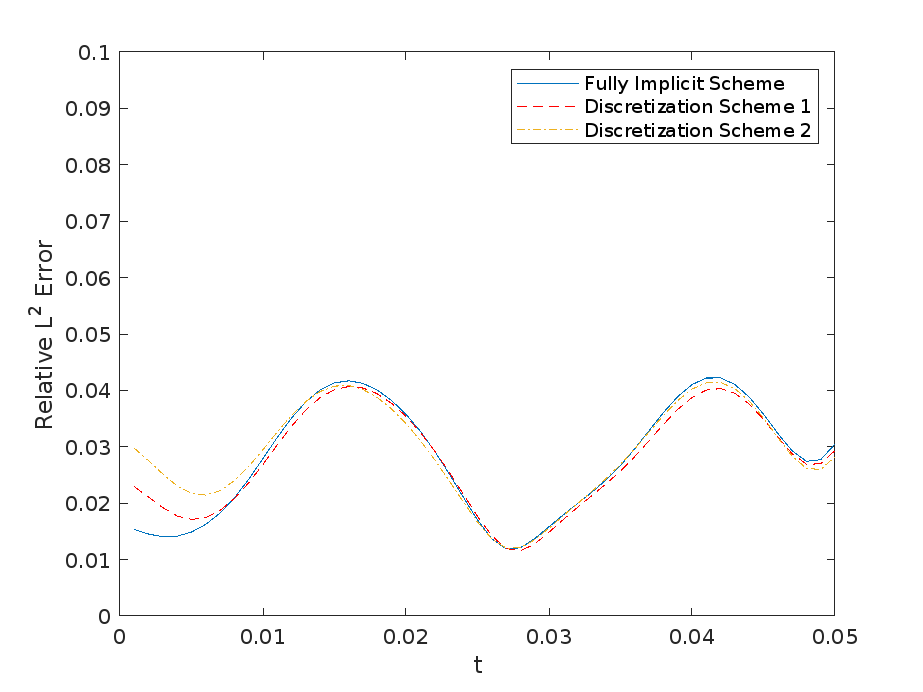}
  \quad
  \caption{Relative $L^2$ error for different schemes when $H=1/10$ and $l=5$ in Example 2. Left: $e^{(1)}(t)$. Right: $e^{(2)}(t)$.}
  \label{fig:Example_2_error_1}
\end{figure}

\begin{figure}
  \centering
  \includegraphics[width=7cm]{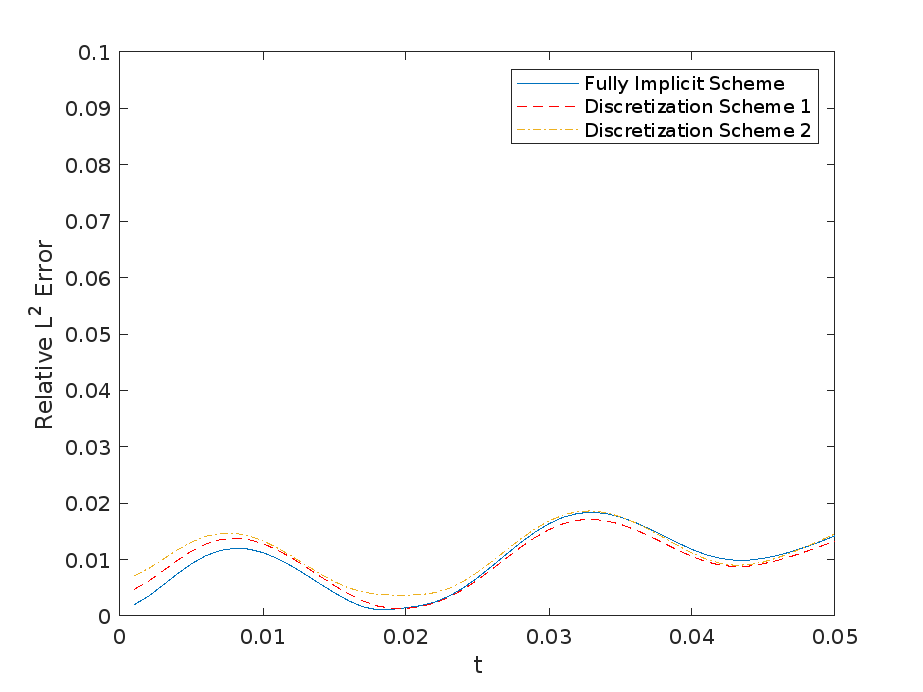}
  \quad
  \includegraphics[width=7cm]{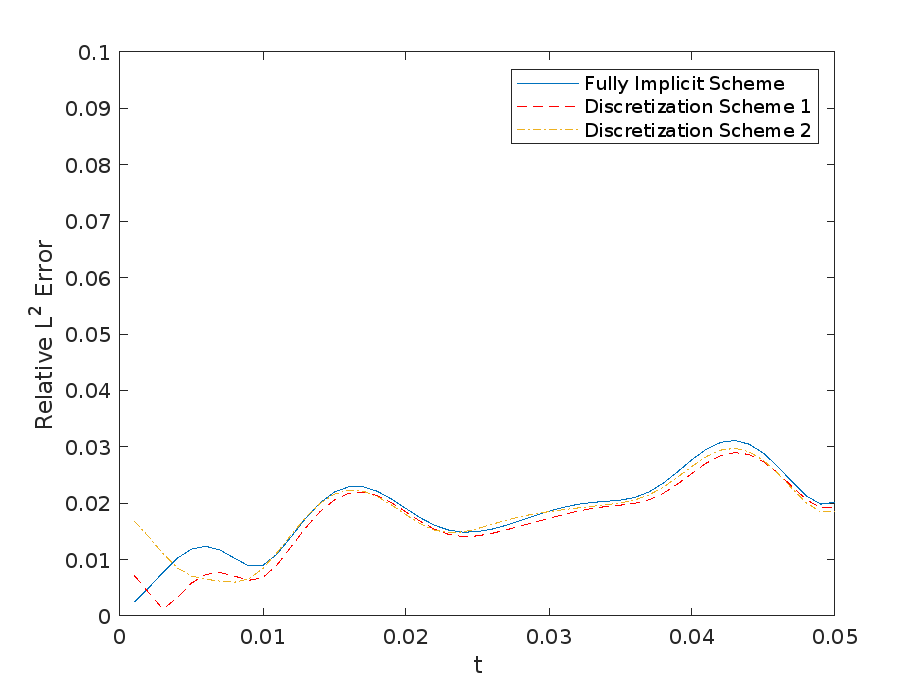}
  \quad
  \caption{Relative $L^2$ error for different schemes when $H=1/20$ and $l=6$ in Example 2. Left: $e^{(1)}(t)$. Right: $e^{(2)}(t)$.}
  \label{fig:Example_2_error_2}
\end{figure}

\subsection{Example 3: Layered field with three continua}

In the following two examples, we consider the multiple-continuum case. Here, we examine a three-value layered field $\kappa$ defined as
\begin{equation}
\kappa(x)=
\begin{cases}
1,&\quad x \in \Omega_1, \\
10^{3}, & \quad x \in \Omega_2, \\
10, & \quad x \in \Omega_3, \\
\end{cases}
\end{equation}
where $\Omega_1$, $\Omega_2$ and $\Omega_3$ correspond to the dark blue, yellow and light blue regions shown in Figure \ref{fig:Example_3_setting}. We choose the auxiliary functions $\psi_1$, $\psi_2$ and $\psi_3$ as the characteristic functions in the above three regions respectively; that is, $\psi_1(x) = \mathds{1}_{\Omega_1}(x)$, $\psi_2(x) = \mathds{1}_{\Omega_2}(x)$ and $\psi_3(x) = \mathds{1}_{\Omega_3}(x)$, in each coarse block.

We then apply the generalized eigenvalue decomposition. The results are listed in Table \ref{tab:Example_3_eigen}. It is observed that the second coordinates of the eigenvectors paired with the two smaller eigenvalues are almost zero. This suggests that the continua associated with the relatively low-value regions are considered to represent slow modes and are handled explicitly, while the continuum associated with the high-value regions is considered to represent fast modes and is handled implicitly. These results align with expectations.

Using the solution space decomposition derived from these results, we compute numerical solutions under different schemes. The reference solution at time $T$ is depicted in Figure \ref{fig:Example_3_setting}. The relative $L^2$ errors are presented in Figures \ref{fig:Example_3_error_1} and \ref{fig:Example_3_error_2}. The figures on the left, middle and right illustrate the relative errors in the low-, high- and medium-value regions respectively. It can be noted that the partially explicit schemes, which require less computational effort, achieve accuracy comparable to the implicit scheme. In contrast, the explicit scheme, not shown in the figures, diverges immediately for the given time step size. Furthermore, increased accuracy can be attained by using a finer coarse mesh size $H$.

\begin{figure}
    \centering
    \includegraphics[width=7cm]{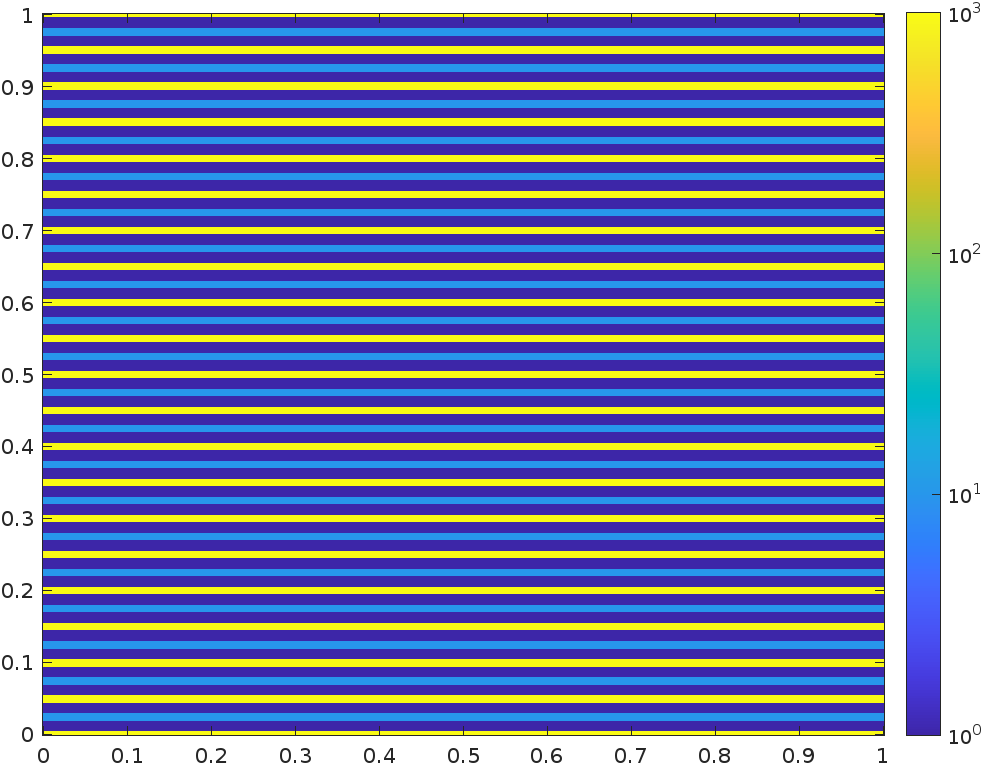}
    \quad
    \includegraphics[width=7cm]{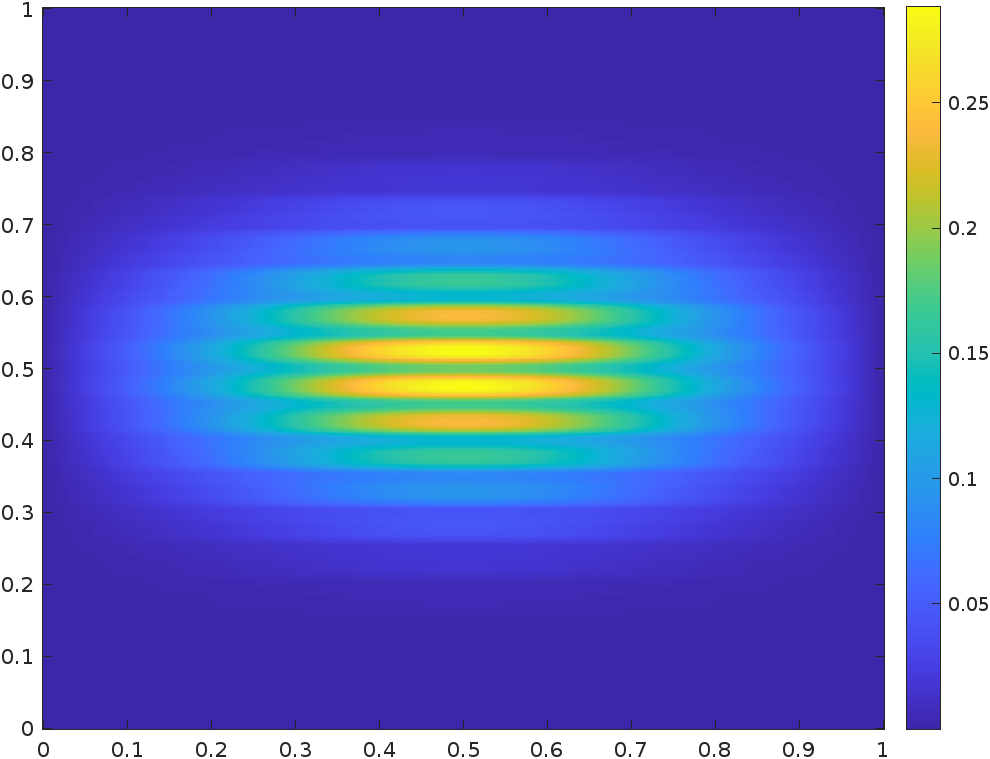}
    \caption{Left: Three-continuum field $\kappa$ in Example 3. Right: Reference solution at the final time $T$ in Example 3.}
    \label{fig:Example_3_setting}
\end{figure}

\begin{table}
    \centering
    \caption{Eigenvalues and corresponding eigenvectors of eigenvalue problems for different coarse mesh sizes $H$ in Example 3.}
    \begin{tabular}{|c|c|c|c|}
         \hline
         $H$ & $l$ & eigenvalues $\lambda_i$ &  eigenvectors $v_i$ \\
         \hline
         \multirow{3}*{1/10} & \multirow{3}*{5} & $10.2379$ & $(1.1485, -0.0034, -0.2543)^T$ \\
         \cline{3-4}
         & & $20.3101$ & $(0.3039, -0.0057, 1.9040)^T$ \\
         \cline{3-4}
         & & $743.7991$ & $(0.1841, 1.9078, 0.1748)^T$ \\
         \hline
         \multirow{3}*{1/20} & \multirow{3}*{6} & $5.1658$ & $(1.1599, -0.0016, -0.1782)^T$ \\
         \cline{3-4}
         & & $12.3465$ & $(0.2580, -0.0054, 1.9126)^T$ \\
         \cline{3-4}
         & & $733.0761$ & $(0.1829, 1.9078, 0.1749)^T$ \\
         \hline
    \end{tabular}   
    \label{tab:Example_3_eigen}
\end{table}

\begin{figure}
  \centering
  \includegraphics[width=5cm]{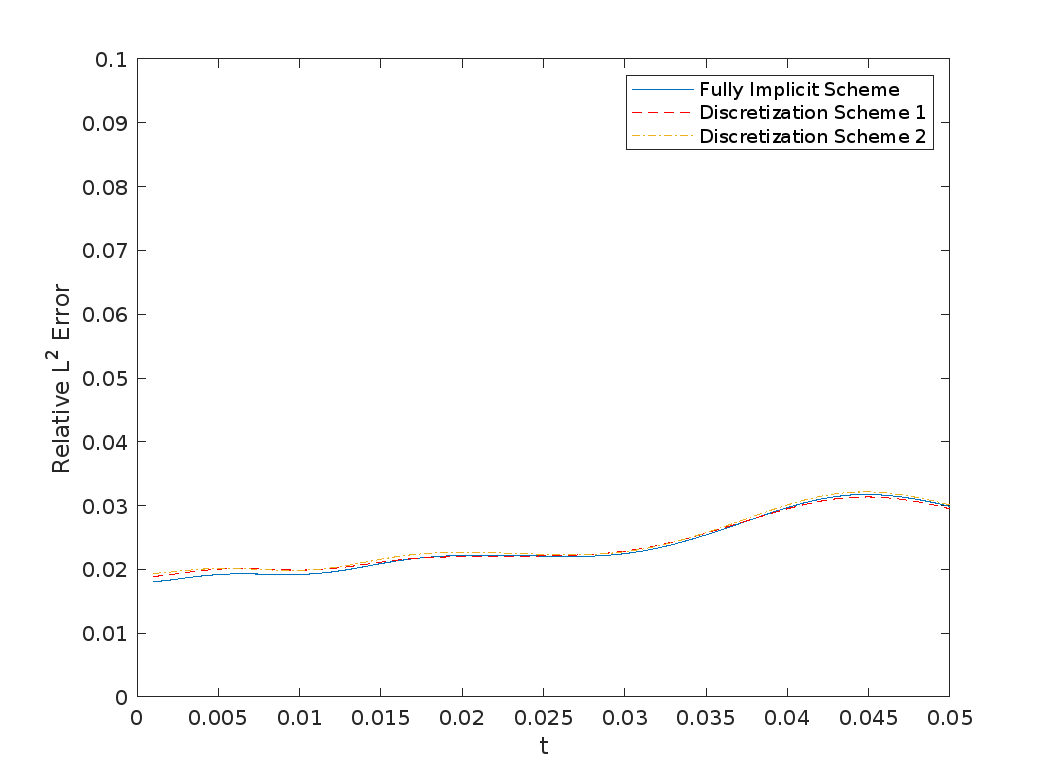}
  \quad
  \includegraphics[width=5cm]{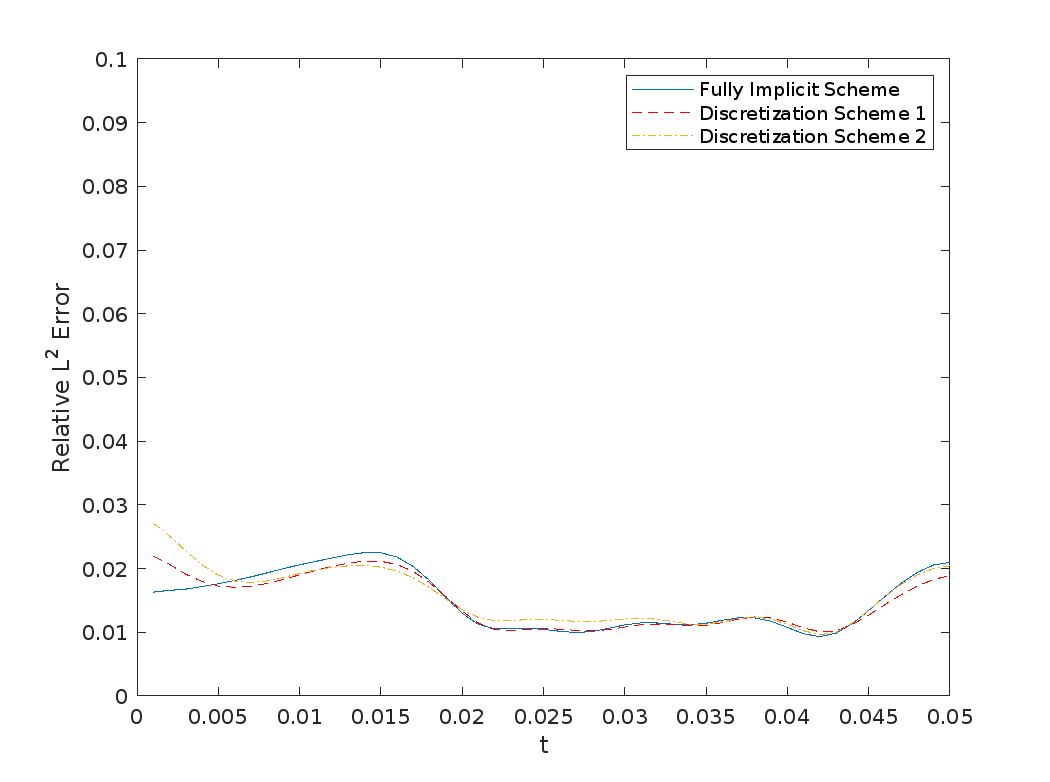}
  \quad
  \includegraphics[width=5cm]{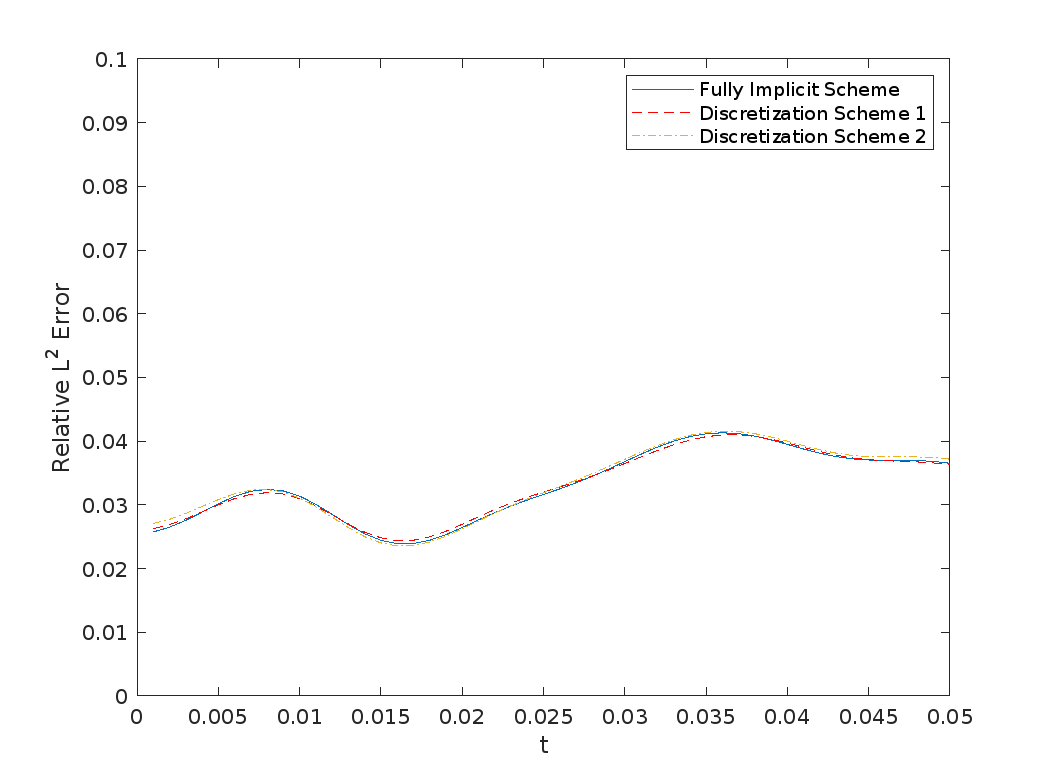}
  \caption{Relative $L^2$ error for different schemes when $H=1/10$ and $l=5$ in Example 3. Left: $e^{(1)}(t)$. Middle: $e^{(2)}(t)$. Right: $e^{(3)}(t)$.}
  \label{fig:Example_3_error_1}
\end{figure}

\begin{figure}
  \centering
  \includegraphics[width=5cm]{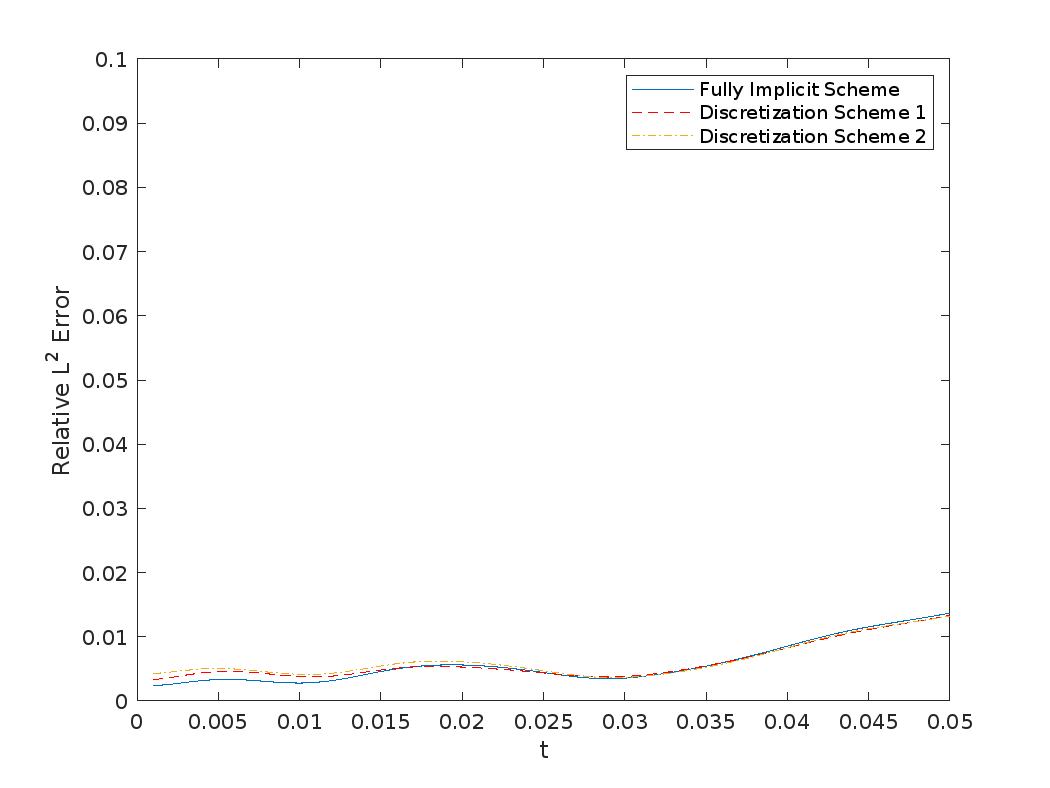}
  \quad
  \includegraphics[width=5cm]{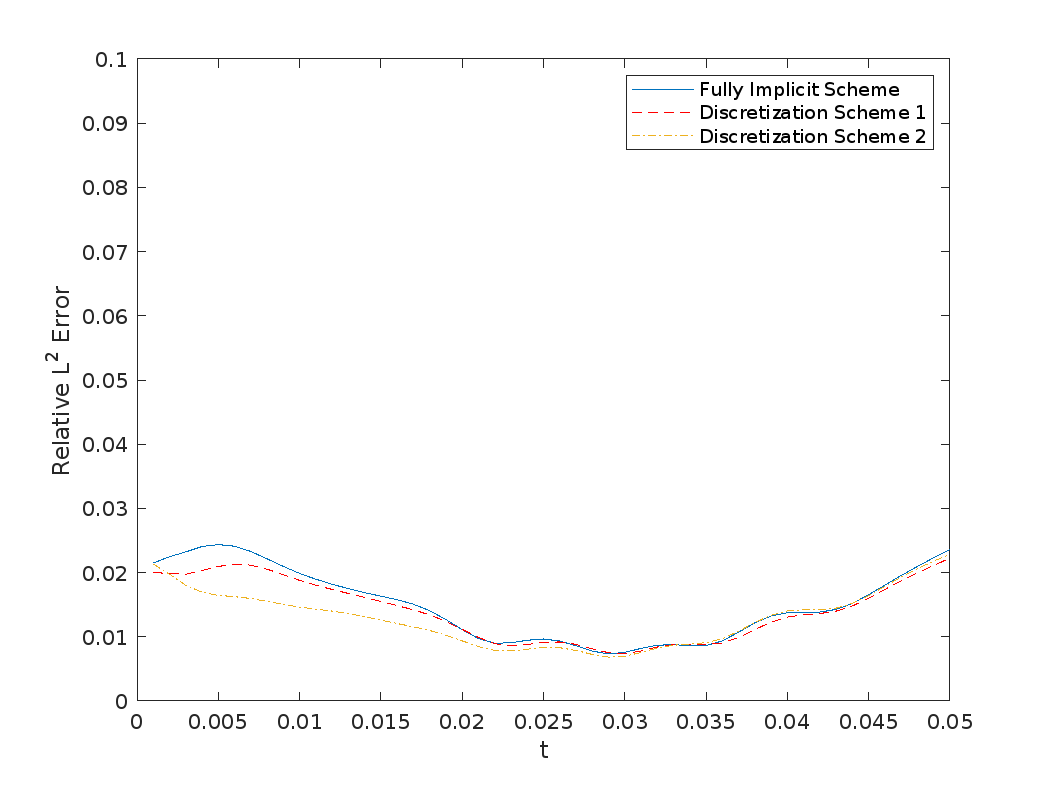}
  \quad
  \includegraphics[width=5cm]{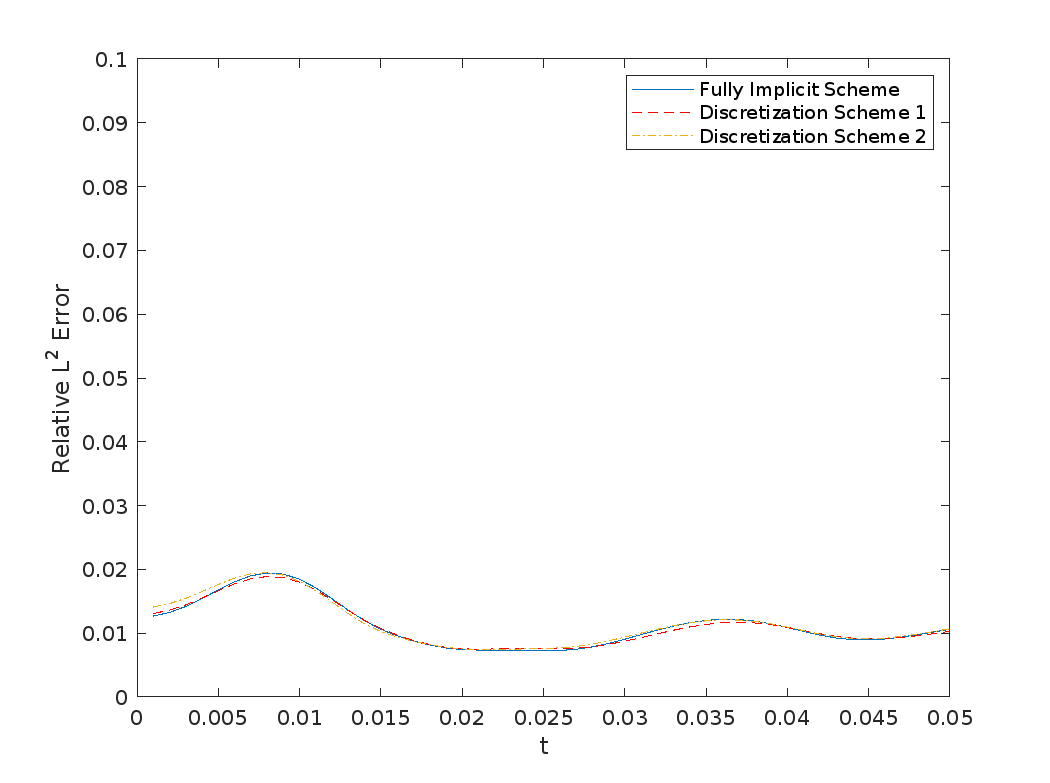}
  \caption{Relative $L^2$ error for different schemes when $H=1/20$ and $l=6$ in Example 3. Left: $e^{(1)}(t)$. Middle: $e^{(2)}(t)$. Right: $e^{(3)}(t)$.}
  \label{fig:Example_3_error_2}
\end{figure}

\subsection{Example 4: Point field with three mixed continua}

In this final example, we consider the coefficient field $\kappa$ as illustrated in Figure \ref{fig:Example_4_setting}, where the lowest value is $1$, the highest value is $10^3$, and the medium value is $4$. The three regions are respectively denoted by $\Omega_1$, $\Omega_2$ and $\Omega_3$. To demonstrate different splitting results, we define the continua in a new pattern. Specifically, in each coarse block, we set $\psi_1(x) = \mathds{1}_{\Omega_1 \cup \Omega_2}(x)$, $\psi_2(x) = \mathds{1}_{\Omega_1 \cup \Omega_3}(x)$, and $\psi_3(x) = \mathds{1}_{\Omega_2 \cup \Omega_3}(x)$. This means that the auxiliary functions represent the characteristic functions in the low-high-value, low-medium-value and medium-high-value regions, respectively. 

The results of the eigenvalue problems are presented in Table \ref{tab:Example_4_eigen}. We notice that our method can effectively separate the slow modes, corresponding to the two smaller eigenvalues, from the overall dynamics by taking linear combinations of the continua. These slow modes will be treated explicitly to reduce the computational cost. We remark that, due to the properties of the coefficient field, the critical eigenvalue in this example is higher than in Example 3, and therefore a finer time step size should be chosen to satisfy the stability condition. However, the previously chosen time step size $\tau$ is relatively loose and remains effective in this case.

The relative $L^2$ errors for different discretization schemes are shown in Figures \ref{fig:Example_4_error_1} and \ref{fig:Example_4_error_2}. Similar observations can be made as in Example 3. 

\begin{figure}
    \centering
    \includegraphics[width=7cm]{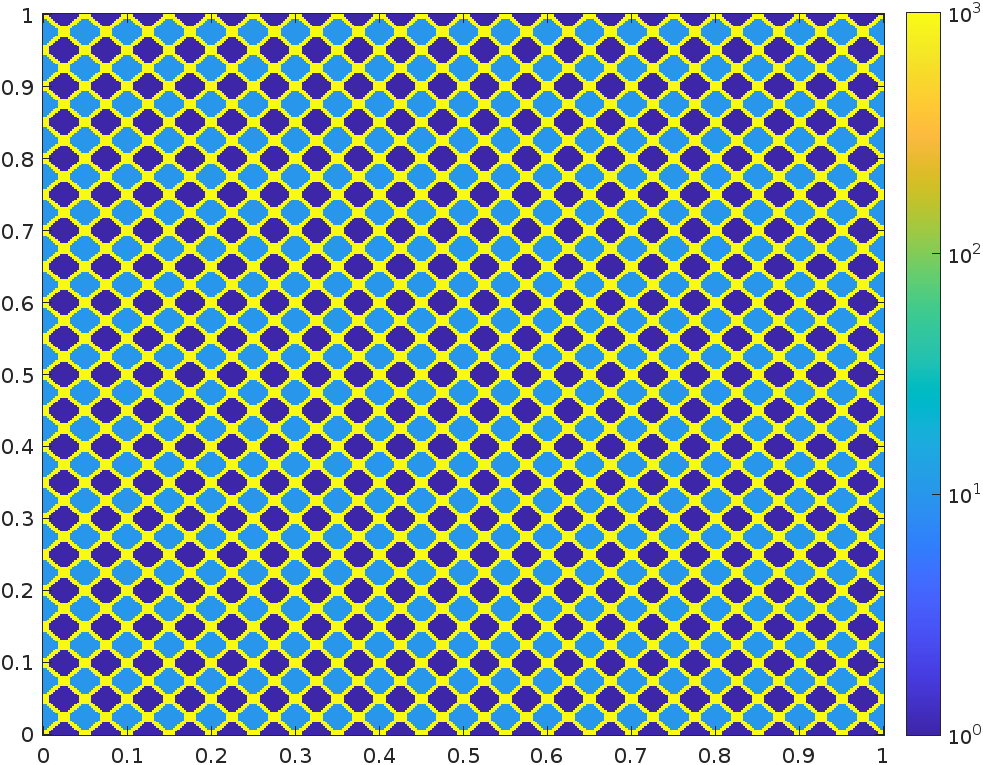}
    \quad
    \includegraphics[width=7cm]{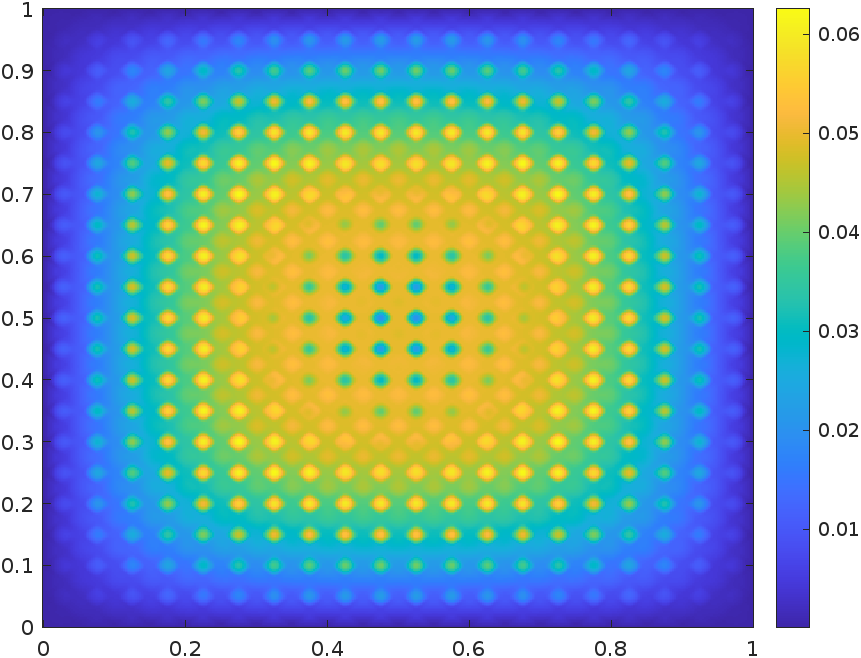}
    \caption{Left: Three-continuum field $\kappa$ in Example 4. Right: Reference solution at the final time $T$ in Example 4.}
    \label{fig:Example_4_setting}
\end{figure}

\begin{table}
    \centering
    \caption{Eigenvalues and corresponding eigenvectors of eigenvalue problems for different coarse mesh sizes $H$ in Example 4.}
    \begin{tabular}{|c|c|c|c|}
         \hline
         $H$ & $l$ & eigenvalues $\lambda_i$ &  eigenvectors $v_i$ \\
         \hline
         \multirow{3}*{1/10} & \multirow{3}*{5} & $40.5777$ & $(-0.4106, 0.1255, 0.6286)^T$ \\
         \cline{3-4}
         & & $144.3006$ & $(0.2038, 0.7969, 0.1353)^T$ \\
         \cline{3-4}
         & & $575.3214$ & $(1.0571, 0.7424, 0.9558)^T$ \\
         \hline
         \multirow{3}*{1/20} & \multirow{3}*{6} & $19.0248$ & $(-0.4488, 0.0985, 0.5934)^T$ \\
         \cline{3-4}
         & & $78.6168$ & $(0.3858, 0.9170, 0.3071)^T$ \\
         \cline{3-4}
         & & $442.8907$ & $(0.9887, 0.5928, 0.9384)^T$ \\
         \hline
    \end{tabular}   
    \label{tab:Example_4_eigen}
\end{table}

\begin{figure}
  \centering
  \includegraphics[width=5cm]{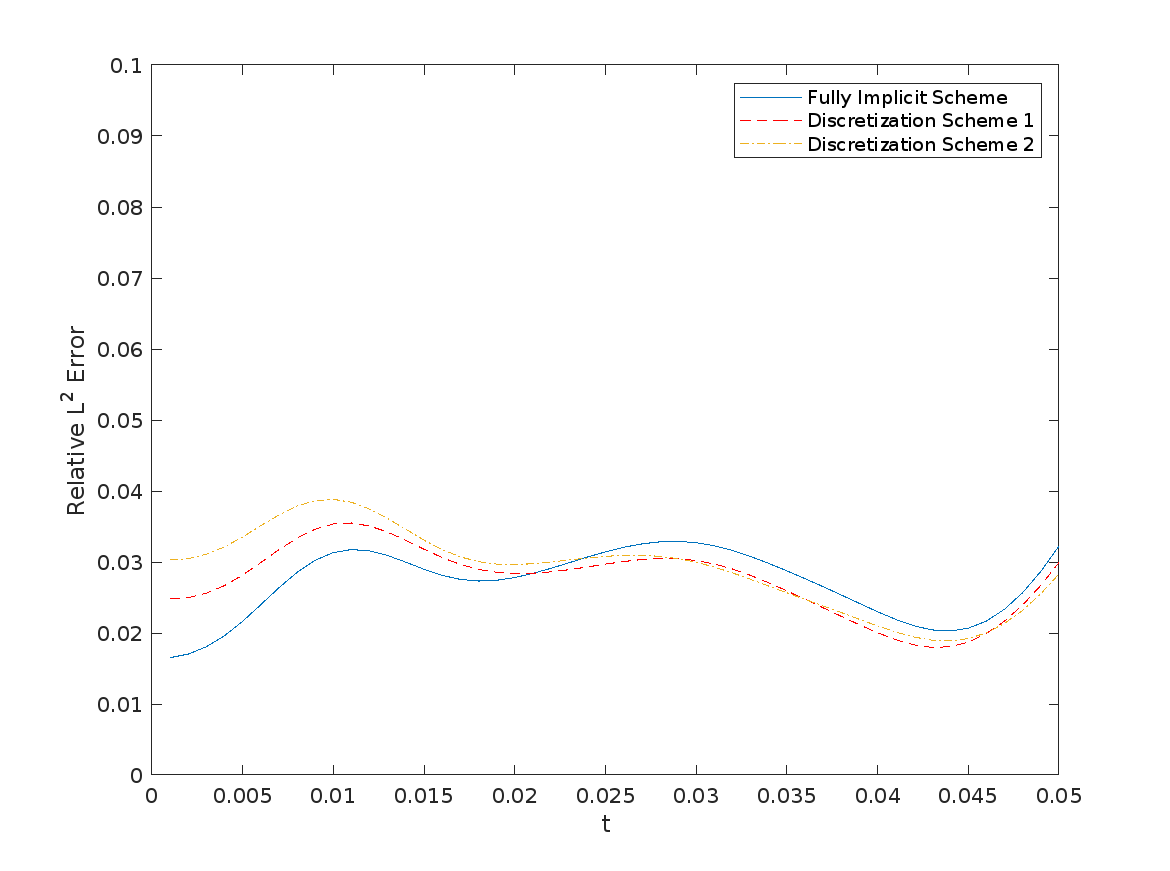}
  \quad
  \includegraphics[width=5cm]{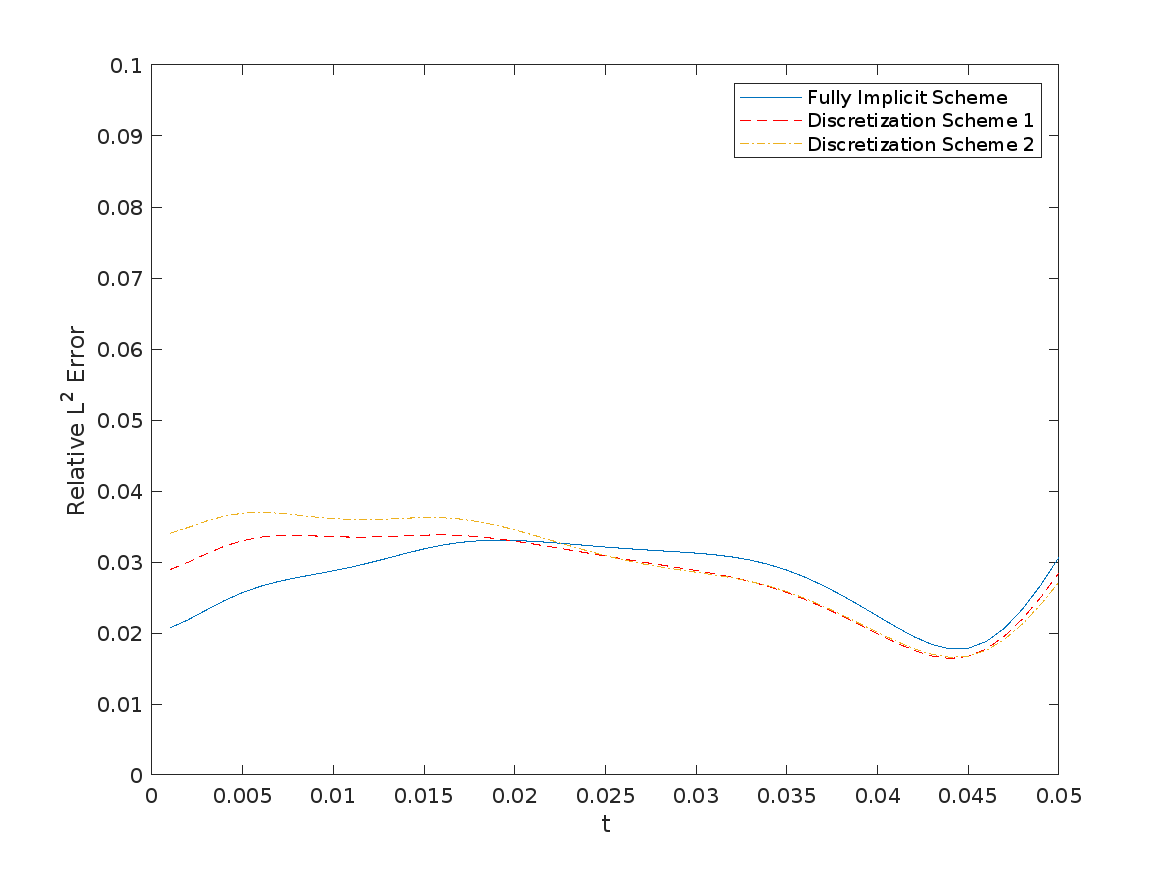}
  \quad
  \includegraphics[width=5cm]{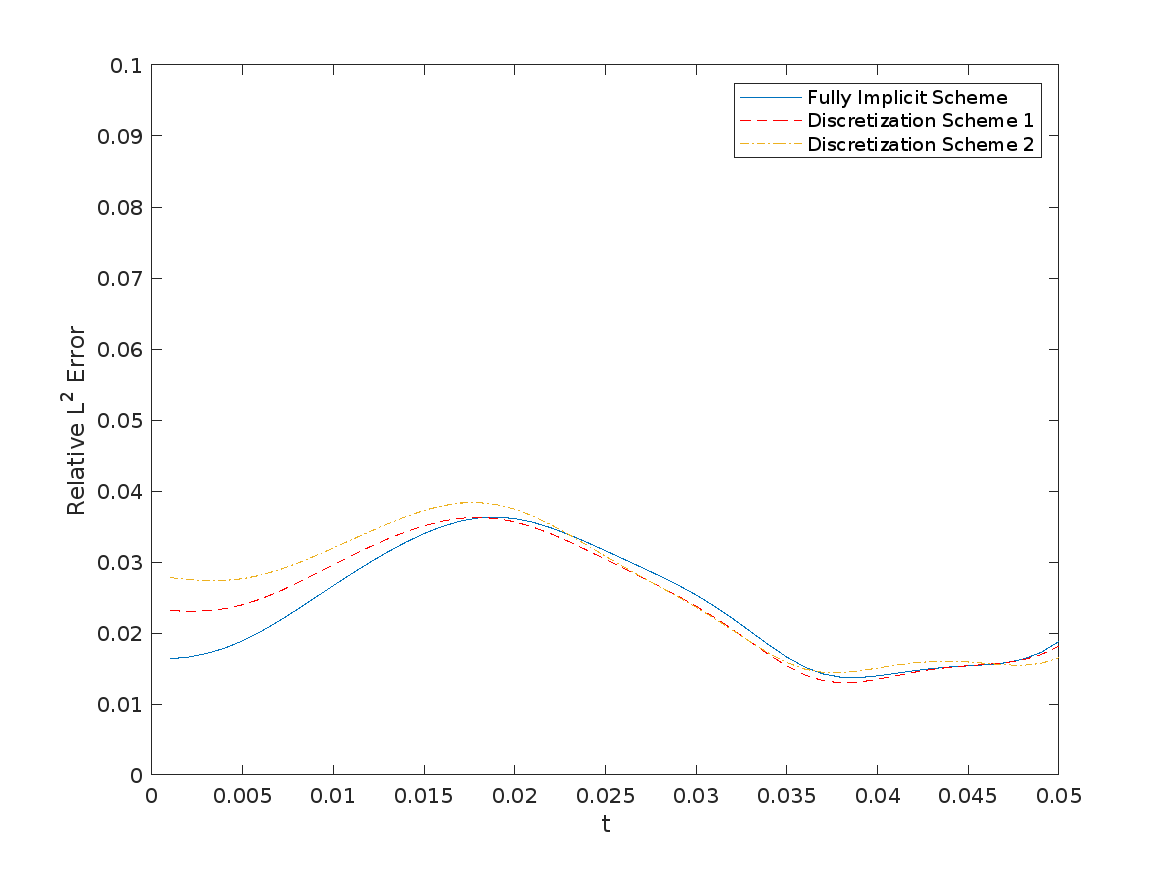}
  \caption{Relative $L^2$ error for different schemes when $H=1/10$ and $l=5$ in Example 4. Left: $e^{(1)}(t)$. Middle: $e^{(2)}(t)$. Right: $e^{(3)}(t)$.}
  \label{fig:Example_4_error_1}
\end{figure}

\begin{figure}
  \centering
  \includegraphics[width=5cm]{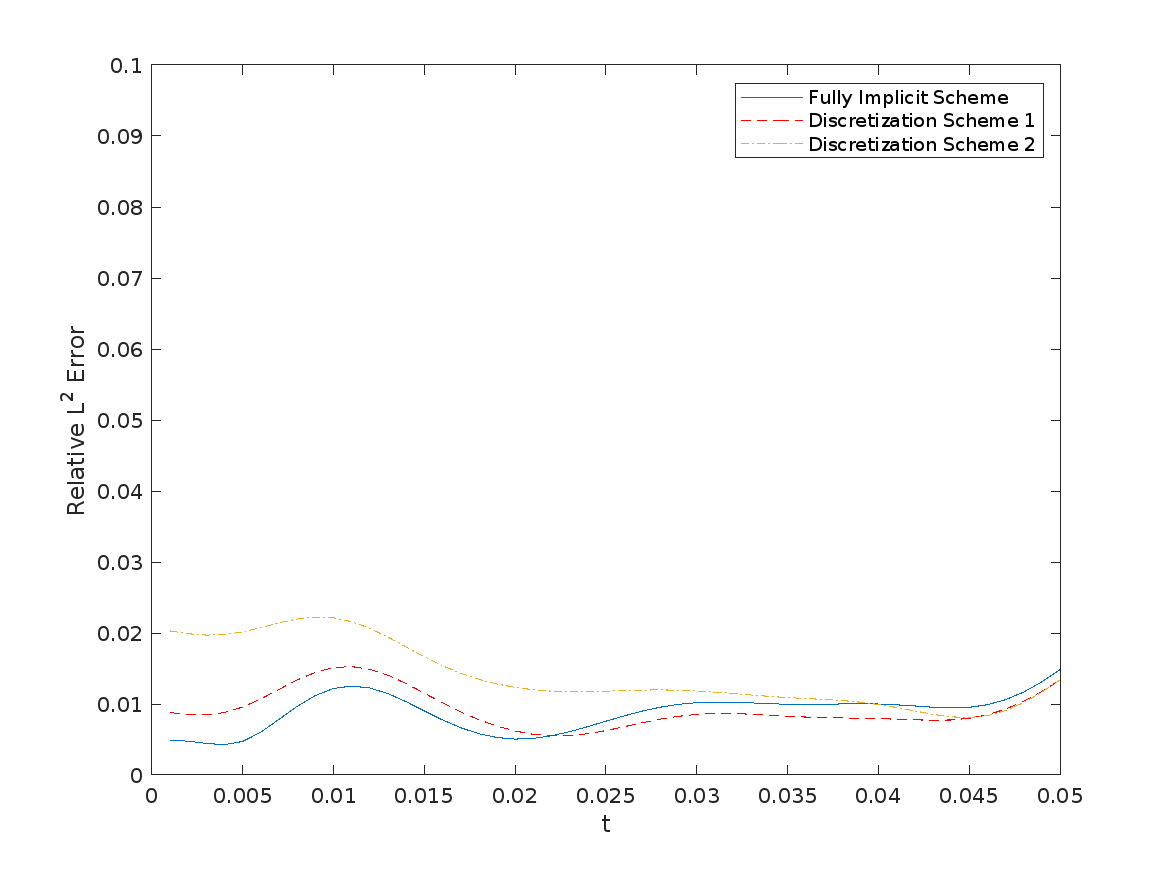}
  \quad
  \includegraphics[width=5cm]{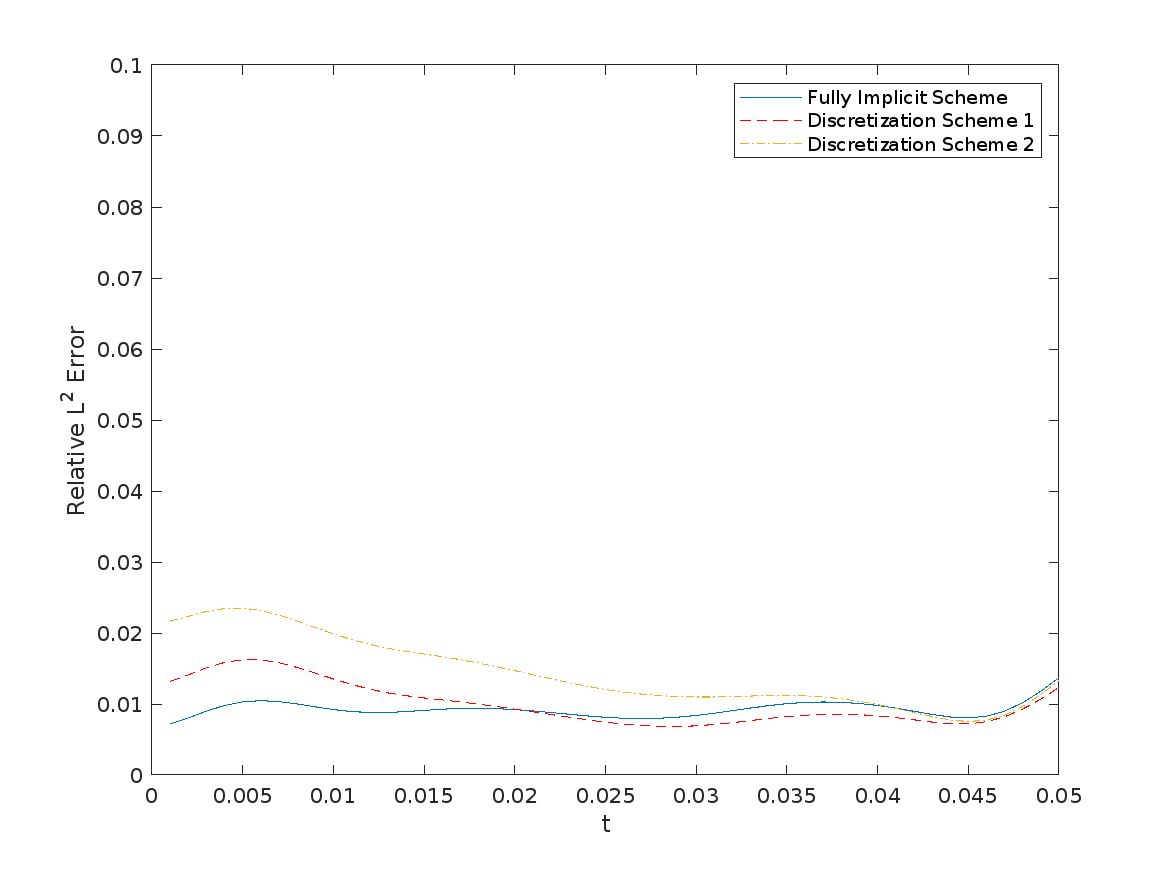}
  \quad
  \includegraphics[width=5cm]{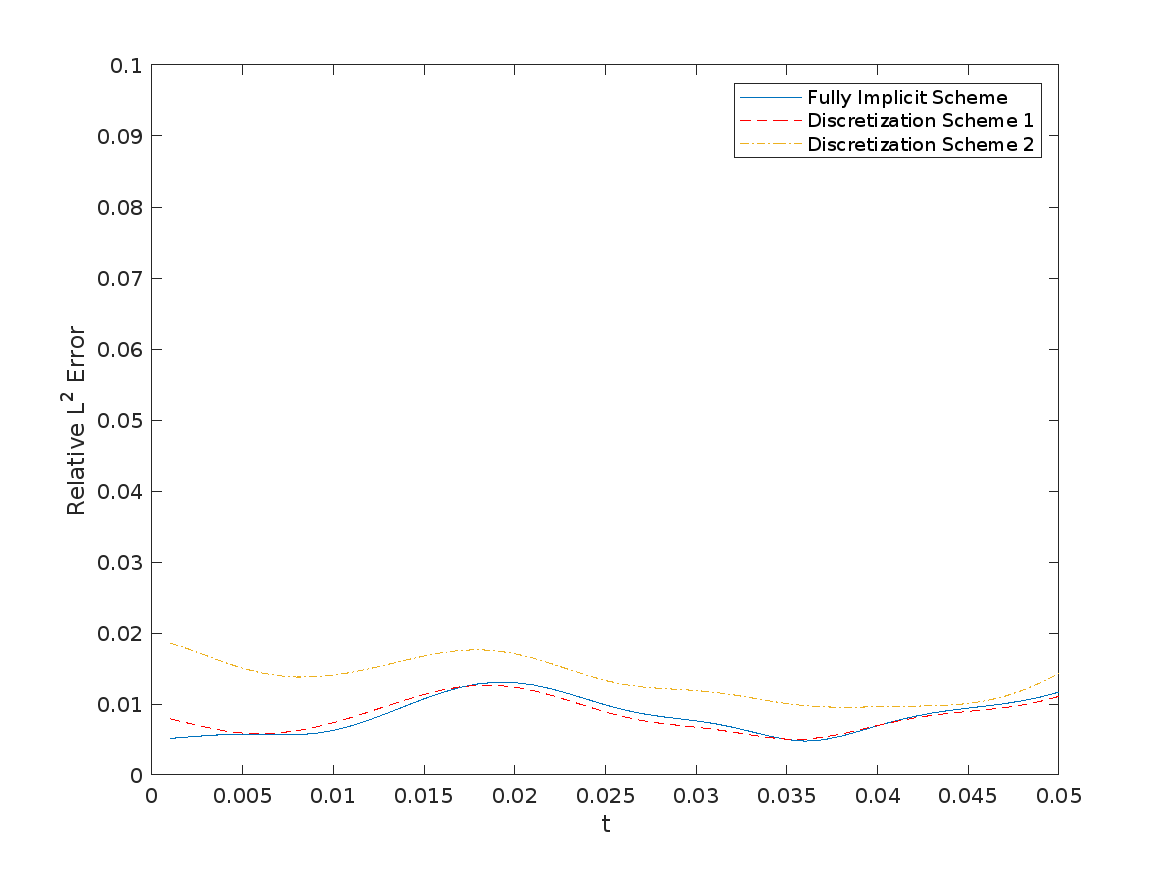}
  \caption{Relative $L^2$ error for different schemes when $H=1/20$ and $l=6$ in Example 4. Left: $e^{(1)}(t)$. Middle: $e^{(2)}(t)$. Right: $e^{(3)}(t)$.}
  \label{fig:Example_4_error_2}
\end{figure}

\section{Conclusions}
\label{sec:conclusions}

In this work, we propose multicontinuum splitting schemes for the wave equation with a high-contrast coefficient. This is an extension of our previous work on multiscale flow problems, including new time discretization schemes, the concept of discrete energy along with its conservation property, and a stability proof derived from this property.

Our approach is built upon the framework of multicontinuum homogenization and draws inspiration from splitting methods. 
First, we decompose the solution space into two components to separate fast and slow dynamics in the system. This decomposition is achieved by introducing physically meaningful macroscopic variables and utilizing the expansion in multicontinuum homogenization. 
Then in the formulated partially explicit time discretization schemes, the fast-dynamics (contrast-dependent) component is treated implicitly to ensure stability, while the slow-dynamics (contrast-independent) component is treated explicitly for computational efficiency. We introduce the concept of discrete energy and derive the stability conditions, which are independent of contrast when the continua are appropriately chosen. Additionally, strategies for optimizing the space decomposition are discussed.

Numerical examples are presented to validate the accuracy and stability of the proposed schemes. The results demonstrate that our proposed approach effectively separates the dynamics with different speeds, and achieves accuracy comparable to that of the implicit scheme. It balances the trade-off between accuracy and efficiency, offering greater stability than the explicit method and requiring less computationally effort than the implicit method.

\bibliographystyle{unsrt}
\bibliography{ref}

\end{document}